\documentclass[a4paper,12pt, oneside]{amsart}
\usepackage{graphicx,amssymb,amstext,amsmath,amsfonts,amsthm}
\usepackage{latexsym}
\usepackage{bbm}
\usepackage{xypic}
\usepackage{nicefrac}
 \usepackage{setspace}
\usepackage{atbegshi}
\usepackage{color}
\usepackage{anysize}
\marginsize{3cm}{2cm}{3cm}{2cm}
\usepackage{setspace}
\newtheorem{theorem}{Theorem}[section]
\newtheorem{lemma}[theorem]{Lemma}
\newtheorem{proposition}[theorem]{Proposition}

\newtheorem{corollary}[theorem]{Corollary}

\theoremstyle{definition}
\newtheorem{definition}[theorem]{Definition}

\theoremstyle{remark}

\numberwithin{equation}{section}
\newcommand\norm[1]{\left\lVert#1\right\rVert_{\mathbb{TV}}}

\newcommand{\ud}{\mathrm{d}}
\newcommand{\cut}{{\mathrm{cut}}}
\DeclareMathOperator{\diag}{diag}

\usepackage{color}
\title[Cut-off phenomenon for  OU processes driven by L\'evy processes]{Cut-off phenomenon for  Ornstein-Uhlenbeck processes driven by L\'evy processes}
\author{Barrera, Gerardo}
\address{
University of Alberta, Department of Mathematical and Statistical Sciences. Central Academic Building. $116$ Street and $85$ Avenue. Postal Code: T$6$G-$2$G$1$. Edmonton, Alberta, Canada.}
\email{barrerav@ualberta.ca }
\author{Pardo, Juan Carlos}
\address{
 CIMAT. Jalisco S/N, Valenciana, CP $36240$. Guanajuato, Guanajuato, Mexico.}
\email{jcpardo@cimat.mx }
\keywords{Cut-off phenomenon, Ornstein-Uhlenbeck processes, L\'evy processes, Total variation distance.}
\date{\today}
\subjclass[2000]{60G51, 60G60, 60G10, 60E07, 37A25}
\begin{document}
\maketitle
\begin{abstract}
In this paper, we study the cut-off phenomenon  under the total variation distance of $d$-dimensional Ornstein-Uhlenbeck processes which are driven by L\'evy processes. That is to say, under the total variation distance, there is an abrupt convergence  of the aforementioned process to its equilibrium, i.e. limiting distribution. Despite that the limiting distribution is not explicit, its distributional properties allow us to deduce that a profile function always exists in the reversible cases and it may exist in the non-reversible cases under suitable conditions on the limiting distribution. The cut-off phenomena for the average and superposition processes are also determined. 
\end{abstract}

\section{Introduction}
 The term {\it cut-off phenomenon}   was introduced by  Aldous and  Diaconis  \cite{AD} in the early eighties to describe the phenomenon of drastic convergence to the equilibrium of  Markov chains models related to  card shuffling. 
Although the {\it cut-off phenomenon} has mostly been discussed in the literature for Markov chains with finite state space, it makes perfect sense in the general  
 context of  Markov processes having limiting distribution. To be more precise, let  us consider $(X^{(\epsilon)}, \epsilon>0)$ a parametrised family of stochastic processes such that for each $\epsilon>0$, the process $X^{(\epsilon)}$ possesses a limiting distribution, here denoted by $\mu^{(\epsilon)}$. Roughly speaking the {\it cut-off phenomenon} refers to the following asymptotic behaviour: as $\epsilon$ decreases,  a suitable distance between the laws of $X^{(\epsilon)}_t$ and the corresponding  limiting distribution $\mu^{(\epsilon)}$ converges to a step function centered at deterministic times $t^{\cut}_\epsilon$.  In other words,   the function $\epsilon\mapsto t^\cut_{\epsilon}$ is such that the distance is asymptotically maximal for times smaller than $t^\cut_\epsilon-o(t^\cut_\epsilon)$ and asymptotically zero for times larger than $t^\cut_\epsilon+o(t^\cut_\epsilon)$.
 
Limiting distributions of stochastic processes are an important feature in probability theory and mathematical physics; and typically  they are not so easy to describe explicitly.  The cut-off phenomenon can be used to simulate such limiting distributions.  More precisely,  the cut-off time $t^\cut_\epsilon$ determines the steps needed to converge to the limiting distribution  within an acceptable error (this will be discussed in more detail below).  Also,  the cut-off phenomenon   can be used  to understand more complicated phenomena which are important in mathematical physics such as metastability (see for instance Barrera et al. \cite{BY2, BY1}).

For an introduction to  the subject in  the Markov chain setting, we refer to   Diaconis \cite{DI},  Mart\'inez and Ycart \cite{MY} and the monograph of Levin et al. \cite{LPW}. We also refer to Saloff-Coste \cite{SAL} for a review of random walks where such  phenomenon appears.  Chen and Saloff-Coste \cite{CSC} considered the cut-off phenomenon for some ergodic Markov processes such as Brownian motions on a compact  Riemann manifold and $k$-regular expander graphs. Lachaud \cite{BL} and Barrera \cite{BA} considered the case of  Ornstein-Uhlenbeck processes driven by a Brownian motion and more recently Barrera and Jara (see \cite{BJ, BJ1}) studied the case of  random dynamical systems with coercive vector fields. 
 
The authors in  \cite{BJ, BJ1} considered the following Langevin dynamics  which are described by the stochastic differential equation (SDE for short), 
\begin{equation}\label{over}
\left\{
\begin{array}{r@{\;=\;}l}
\ud X^{(\epsilon)}_t&-F(X^{(\epsilon)}_t)\ud t + \sqrt{\epsilon}\ud B_t\quad  
\textrm{  for }\quad t\geq 0,\\
X^{(\epsilon)}_0 & x_0\in\mathbb{R}^d\setminus\{0\},
\end{array}
\right.
\end{equation}
where $\epsilon>0$, which can be considered as the amplitude of the noise, $F$ is a strongly coercive vector field with a unique asymptotically stable attractor at $0$ and satisfying an exponential growth condition; and $(B_t,t\geq 0)$ is a standard Brownian motion in $\mathbb{R}^d$.  Under the above assumptions, such SDEs converge to their equilibrium and also exhibits, under the total variation distance,  the cut-off phenomenon. Moreover, in the unidimensional case and assuming that the dynamics are given by a multi-well potential, the cut-off phenomenon allows  to describe the densities of the quasi-stationary measures which are associated to the metastability phenomena (see Section $5$ in \cite{BJ} for further details about this fact).\\

In this manuscript, we are interested in the cut-off phenomenon of  Ornstein-Uhlenbeck processes driven by L\'evy processes (or OUL processes). The latter  are defined  as the unique strong solution of the following SDE
\begin{equation}\label{first}
\left\{
\begin{array}{r@{\;=\;}l}
\ud X^{(\epsilon)}_t &-Q X^{(\epsilon)}_t\ud t+\sqrt{\epsilon}\ud \xi_t\quad  
\textrm{  for } \quad t>0,\\
X^{(\epsilon)}_0 & x_0\in\mathbb{R}^d\setminus\{0\},
\end{array}
\right.
\end{equation}
where  $Q$ is a $d$-squared  real matrix whose eigenvalues have positive real parts, $\xi=(\xi_t,t\geq 0)$ denotes a $d$-dimensional L\'evy process and $\epsilon>0$.  Such class of processes are also  known as processes of the Ornstein-Uhlenbeck type, according to Sato and Yamazato \cite{SaYa} terminology,  or the L\'evy driven case of a generalised Ornstein-Uhlenbeck process, according to Kevei \cite{PK} and the references therein.

The study of this particular case is relevant for  the understanding of  the cut-off and metastability phenomena of a wide class  of SDEs driven by L\'evy noises in $\mathbb{R}^d$ and, moreover, it also shows the complexity that brings to  the problem  the addition of a Poisson jump structure with respect to the Brownian case. For instance, by replacing the Brownian component by  a L\'evy process  in \eqref{over} and since the vector field is  strongly coercive, a natural  approach to deduce the cut-off phenomenon for the family of SDEs  is via  a linearisation technique (similar to \cite{BJ, BJ1})  and henceforth  the  Ornstein-Uhlenbeck case  is needed. Actually, this is the strategy that is used by the authors, together with M. H\"ogele, in a forthcoming manuscript \cite{BHP} where the non-linear case is studied.   Moreover, the jump structure may imply that the positive integers moments may not exist and  even exponential moments, a property which is strongly used in \cite{BJ, BJ1}. This implies that the strategy of the proof may change substantially  with respect to the Brownian case and new techniques and couplings are needed. On the other hand, under the absence of the Brownian component, the transition functions of  the SDEs may be quite irregular (this will be discussed below), therefore some conditions on the jump structure of the L\'evy process are needed.  We also  note that in the Brownian case the transition functions and the limiting distribution are Gaussian distributions and implicitly good bounds of the total variation distance  between them can be obtained. This property is strongly used in the papers \cite{BJ, BJ1},  and  is lost when a Poisson jump structure is added to the noise.

Under the assumption  that the jump structure of the L\'evy process $\xi$ has finite log-moment, 
as well as  some regularity conditions that we will specify below, we get the cut-off phenomenon for the family of OUL $(X^{(\epsilon)}, \epsilon>0)$ where $X^{(\epsilon)}$ satisfies \eqref{first}. Both conditions allow us to use a wide class of L\'evy processes which  include, in particular,  stable processes and Brownian motion.
Moreover, we also provide conditions on the limiting distribution for having the so-called profile cut-off, i.e. the abrupt convergence is described by a deterministic function. For instance, this condition is full-filled for  symmetric distributions in the unidimensional case.

It is important to note that in \cite{BJ1} a necessary and sufficient condition is obtained in order to get profile cut-off, something that seems not so easy to deduce when a Poisson jump structure is added to the noise. Indeed the  authors in \cite{BJ1} provided a characterisation of the profile cut-off using 
the invariance property of the standard Gaussian distribution by orthogonal matrices (see Lemma A$.2$ in \cite{BJ1}) and the upper and lower semi-continuity property for the total variation distance between Gaussian distributions (see Lemma A$.6$ in \cite{BJ1}). For tracking this issue, we impose  an ``invariance" type condition and provide an alternative proof that does not rely on  explicit computations of the total variation distance.

Finally, we are also interested  on the cut-off phenomenon for the superposition and the average processes of   OUL.  For the superposition process of OUL,  profile cut-off is also obtained and the profile function is given in terms of a self-decomposable distribution. Motivated by the work of Lachaud \cite{BL}, we study the cut-off phenomenon for the average process of OUL under the assumption that the driving L\'evy process is stable and prove that there is  profile cut-off with an explicit profile function, cut-off time  and window cut-off. 

\section{Preliminaries and main results}

\subsection{Cut-off phenomenon }Before we  introduce the concept of cut-off  formally, let us recall the notion of the total variation distance which will be our reference distance between probability distributions.
Given two probability measures $\mathbb{P}$ and $\mathbb{Q}$ which are defined in the same measurable space $\left(\Omega,\mathcal{F}\right)$,
the total variation distance between $\mathbb{P}$ and $\mathbb{Q}$ is given by
\[
\norm{\mathbb{P}-\mathbb{Q}}:=\sup\limits_{A\in \mathcal{F}}{|\mathbb{P}(A)-\mathbb{Q}(A)|}.\] 
For simplicity, in the case of two random variables $X$ and $Y$ defined on the same probability space $\left(\Omega,\mathcal{F}, \mathbb{P}\right)$ we use the following notation  for its total variation distance,
$$\norm{X-Y}:=\norm{\mathcal{L}(X)-\mathcal{L}(Y)}, $$
where $\mathcal{L}(X)$ and $\mathcal{L}(Y)$ denote the law under $\mathbb{P}$ of the random variables $X$ and $Y$, respectively. 
For a complete understanding of the total variation distance (normalised or not normalised), we refer to Chapter $2$ of the monograph of Kulik \cite{KUL}.

According to Barrera and Ycart \cite{BY}, the cut-off phenomenon  can
be expressed at three increasingly sharp levels.

\begin{definition}\label{definition}
The family $(X^{(\epsilon)}, \epsilon>0)$ possesses
\begin{itemize}
\item[i)] {\it cut-off} at times $(t_{\epsilon}, \epsilon>0)$,  if $t_{\epsilon}$ goes to $\infty$ accordingly as $\epsilon$ goes to $0$ and
\begin{eqnarray*}
\lim\limits_{\epsilon\rightarrow 0 }{d^{(\epsilon)}(ct_{\epsilon})}= \left\{ \begin{array}{lcc}
             1 &  \textrm{ if }  & 0 < c < 1, \\
             \\ 0 & \textrm{ if } & c>1. \\
             \end{array}
   \right.
\end{eqnarray*}
\item[ii)]  {\it window cut-off} at
$(\left(t_{\epsilon}, w_{\epsilon}\right), \epsilon>0)$, if $t_{\epsilon}$ goes to $\infty$ when $\epsilon$ goes to $0$, $w_{\epsilon}=o\left(t_{\epsilon}\right)$, and
\[
\lim\limits_{c \rightarrow -\infty}{\liminf\limits_{\epsilon\rightarrow 0}
{d^{(\epsilon)}(t_{\epsilon}+cw_{\epsilon})}}=1\qquad \textrm{and} \qquad  \lim\limits_{c \rightarrow \infty}{\limsup\limits_{\epsilon\rightarrow 0}
{d^{(\epsilon)}(t_{\epsilon}+cw_{\epsilon})}}=0.
\]
\item[iii)]  {\it profile cut-off} at
$(\left(t_{\epsilon}, w_{\epsilon}\right), \epsilon>0)$ with profile function $G$, if $t_{\epsilon}$ goes to $\infty$ when $\epsilon$ goes to $0$, $w_{\epsilon}=o\left(t_{\epsilon}\right)$,
\begin{eqnarray*}
 G(c):=\lim\limits_{\epsilon \rightarrow 0}{d^{(\epsilon)}(t_{\epsilon}+cw_{\epsilon})}
\end{eqnarray*} is well-defined for all $c\in \mathbb{R}$ and satisfies
\[
\lim\limits_{c \rightarrow -\infty}{G(c)}=1\qquad \textrm{and} \qquad \lim\limits_{c \rightarrow \infty}{G(c)}=0.
\]
\end{itemize}
\end{definition}
Observe that the cut-off times and the windows cut-off are deterministic and both  may  depend on the starting state of the process. Implicitly, the distance $d^{(\epsilon)}(t)$ also may depend on the starting state. Moreover, the cut-off times and windows cut-off may not be unique but, up to an equivalence relation, they are (see \cite{MY} for further details).
On the other hand, there are not to many examples where the profile can be determined explicitly, specially under the total variation distance. Actually explicit profiles are usually out of reach and normally  only  windows cut-off can be hoped for.

The cut-off phenomenon, under the total variation distance, is naturally associated to a switching phenomenon, {\it i.e.}, {\it all/nothing} or {\it $1$/$0$ behaviour} but it has the drawback of being used with other meanings in statistical mechanics and theoretical physics. Alternative names are {\it threshold phenomenon} and {\it abrupt convergence} (see Barrera et al. \cite{BY1}). The cut-off phenomenon can  be  also interpreted as a mixing time (see for instance Lubetzky and Sly \cite{LS} and/or Chapter $18$ of \cite{LPW})  or as a hitting time  (see for instance \cite{MY}). Both interpretations are equivalent for the total variation distance and separation distance, see Barrera and Ycart \cite{BY} and the references therein.

Let us exemplify the relevance of the cut-off phenomenon by the following simple and well-known example. Imagine that we would like to sample a probability distribution  by a Markov Chain Monte Carlo  method using an ergodic Markov chain  that has the desired distribution as its limiting distribution. Usually, it is not so difficult to construct such Markov chain  with the given properties. The more difficult problem is to determine how many steps are needed to converge to the limiting distribution within an acceptable error.  The cut-off phenomenon, in this setting, implies that there  exists an asymptotically optimal sufficient running time,  here denoted by  $T_\epsilon$, which  is asymptotically equivalent to the cut-off time $t_\epsilon$. Moreover, if there is a cut-off  time 
$t_\epsilon$ with window size $w_\epsilon$, then one gets the more precise result, that is to say,  that the optimal running time $T_\epsilon$ should satisfy $|T_\epsilon-t_\epsilon| = \mathcal{O}(w_\epsilon)$ as $\epsilon$ goes  to $0$. The crucial point here is that, if there is a cut-off, these relations hold for any desired fixed admissible error size whereas, if there is no cutoff, the optimal sufficient running time $T_\epsilon$ depends greatly of the desired admissible error
size. For further details we refer \cite{CSC}.

\subsection{The Ornstein-Uhlenbeck process and its invariant distribution}
Similarly to the diffusive case, OUL  have been widely studied since they appear in many areas of applied probability. This family of  processes  appears as a natural continuous time generalisation of random recurrence equations, as shown by de Haan and Karandikar \cite{HK} and has applications in mathematical finance (see for instance Kl\"uppelberg et al. \cite{KLM} and  Yor \cite{Yor}),  risk theory (see for instance Gjessing and Paulsen \cite{GP}), mathematical physics (see for instance Garbaczewski and Olkiewicz \cite{GO}) and random dynamical systems (see for instance  Friedman \cite{Friedman}).  From the distributional point of view, they have attracted a lot of attention since the limiting distribution, whenever it exists,  satisfies an operator self-decomposability property which in the unidimensional case turns out to be the so-called self-decomposability property, see for instance Sato and Yamazato \cite{SaYa} for further details.  Actually, any operator self-decomposable  distribution can be determined as the equilibrium distribution of an OUL, see for instance Sato and Yamazato \cite{SaYa} and Sato \cite{Sa, Sa1} for the self-decomposable case.

 Let  $d\ge 1$ and  $\xi=(\xi_t, t\geq 0)$ be a $\mathbb{R}^d$-valued L\'evy process, that is to say a c\`adl\`ag
process with independent and stationary increments, whose law, starting from $x\in \mathbb{R}^d$, is denoted  by $\mathbb{P}_x$,  with the understanding that $\mathbb{P}_{0} = \mathbb{P}$.  We also let  $|\cdot |$  and $\langle \cdot,\cdot \rangle$  be the Euclidean norm and the  standard inner product in $\mathbb{R}^d$, respectively.

It is well known that the law  of any L\'evy process  is characterised by its one-time transition probabilities. In particular, for all $z \in\mathbb{R}^d$, we have
\begin{equation*}
\mathbb{E}\Big[e^{i\langle z,\xi_t\rangle}\Big] = e^{ t \psi( z)},
\end{equation*}
where the characteristic exponent $\psi$  satisfies the so-called L\'evy-Khintchine formula
\begin{equation*}\label{expc}
\psi(z) = -\frac{1}{2}\langle z, \Sigma z\rangle +  i\langle a,z\rangle + \int_{\mathbb{R}^d}\left( e^{i \langle z,x\rangle} -1- i\langle z,x\rangle\mathbf{1}_{\{|x|\le 1\}}\right)\nu({\rm d}x),
\end{equation*}
$a\in\mathbb{R}^d$,
$\Sigma$ is a $d$-squared symmetric non-negative definite matrix and $\nu$ is a measure on $\mathbb{R}^d\backslash\{0\}$ satisfying the integrability condition
\[
\int_{\mathbb{R}^d}\left(1\wedge |x |^2\right)\nu({\rm d}x)<\infty.
\]
We take  $\epsilon>0$ and  recall that the associated OUL  $X^{(\epsilon)}$ is the  unique strong solution of the following linear SDE  in $\mathbb{R}^d$ given by
\begin{equation}\label{sdem}
\left\{
\begin{array}{r@{\;=\;}l}
\ud X^{(\epsilon)}_t & -Q X^{(\epsilon)}_t\ud t+\sqrt{\epsilon}\ud \xi_t \qquad \textrm{for}\quad t\geq 0, \\
X^{(\epsilon)}_0 & x_0\not=0,
\end{array}
\right.
\end{equation} 
where $Q$ is a $d$-squared real matrix whose eigenvalues has positive real parts. For simplicity, we denote the latter class of matrices by $\mathcal{M}^{+}(d)$. If the matrix  $Q$ is symmetric, we say that the process \eqref{sdem} is reversible since the vector field $F(x)=Qx$ can be written as the transpose of the gradient for the quadratic form $V(x)=x^{T}Qx/2$,  which can be thought as  potential energy. Here,  $x^T$ denotes the transpose of the vector $x$.
When the matrix $Q$ is non-symmetric, we say that the process \eqref{sdem} is non-reversible.
It is important to note that when we perturb a symmetric matrix, typically it becomes non-symmetric, in other words and  roughly speaking, most of the matrices on $\mathcal{M}^{+}(d)$ are non-symmetric.

The  SDE (\ref{sdem})  can be rewritten as
\[
X^{(\epsilon)}_t =e^{-tQ}x_0+ \sqrt{\epsilon}\int_{0}^t e^{-(t-s)Q}\ud \xi_s \qquad \textrm{for}\quad t\ge 0.
\]We denote by $\mathbf{P}_{x_0}$ for its law starting from $x_0$. It is known (see for instance Theorem  3.1 in Sato and Yamazato \cite{SaYa}) that the latter is an homogeneous Markov process with transition function $\mathtt{P}^{(\epsilon)}_t(x_0, B)=\mathbf{P}_{x_0}(X^{(\epsilon)}_t\in B)$, for $B\in \mathcal{B}(\mathbb{R}^d)$, satisfying
\begin{equation*}\label{care}
\int_{\mathbb{R}^d} e^{i\langle \lambda, y\rangle}\mathtt{P}^{(\epsilon)}_t(x_0, \ud y)=\exp\left\{i\langle x^T_0e^{-tQ^T },\lambda \rangle +\int_0^t \psi(\sqrt{\epsilon} e^{-sQ^T } \lambda)\ud s \right\}, \quad \lambda\in \mathbb{R}^d,
\end{equation*}
where $Q^T$ denotes the matrix transpose of $Q$.

By straightforward computations, we deduce that  the transition function ${\tt P}^{(\epsilon)}_t(x_0, \cdot)$ is infinitely divisible with generating triple $(a^{(\epsilon)}_t, \Sigma^{(\epsilon)}_t, \nu^{(\epsilon)}_t)$ given by
\begin{align*}
 \Sigma^{(\epsilon)}_t&=\epsilon \int_0^t e^{-sQ }\Sigma e^{-sQ^T } \ud s,\qquad \nu^{(\epsilon)}_t(B)=\int_0^t\nu(\sqrt{\epsilon}e^{sQ} B)\ud s \qquad \textrm{for} \quad B\in \mathcal{B}(\mathbb{R}^d) ,\\
 a^{(\epsilon)}_t &=e^{- tQ} x_0+\sqrt{\epsilon}\int_0^t e^{-sQ}a\ud s+\sqrt{\epsilon}\int_{\mathbb{R}^d}\nu(\ud y) \int_0^t e^{-sQ} y\left(\mathbf{1}_{\{|\sqrt{\epsilon}e^{-sQ} y|\le 1\}}-\mathbf{1}_{\{|y|\le1 \}}\right)  \ud s,
\end{align*}
where $e^{sQ} B=\{y\in \mathbb{R}^d: y=e^{sQ}x, x\in B \}$ (see for instance Theorem  $3.1$ in  \cite{SaYa}). Moreover, according to Masuda \cite{MA}, the process has a transition density which is infinitely differentiable and bounded (i.e. belongs to $\mathcal{C}^\infty_b$) if the ${\rm rank}(\Sigma)=d$ or the L\'evy measure $\nu$ satisfies  {\it Orey-Masuda's condition}, that is to say, if there exist constants $\alpha\in (0,2)$ and $c>0$ such that
\begin{equation}\label{Masuda'scond}
\int\limits_{\{z: |\langle v,z \rangle|\le 1\}}|\langle v,z \rangle|^2\nu(\ud z)\ge c|v|^{2-\alpha} 
\qquad \textrm{for any $v\in \mathbb{R}^d$ with $|v|\ge 1$.}
\end{equation}
A weaker assumption, on the L\'evy measure $\nu$, for the transition density  being  infinitely differentiable and bounded appears in Bodnarchuk and Kulyk \cite{BodK, BodK1}. Indeed, according to Theorem 1 in \cite{BodK1}  if 
\begin{equation}\label{BodKcond}
-\frac{1}{r^2\ln(r)}\inf_{\ell \in \mathbb{S}^{d-1}}\int_{\mathbb{R}^d} \Big(\langle z, \ell\rangle^2\land r^2\Big)\nu(\ud z)\to \infty
\qquad \textrm{as $r$ goes to $0$,}
\end{equation} 
where  $\mathbb{S}^{d-1}$ denotes the unit sphere in $\mathbb{R}^d$, then the transition density  belongs to $\mathcal{C}^\infty_b$. Actually in the one dimensional case, the following condition 
\begin{equation}\label{BodK1cond}
-\frac{1}{r^2\ln(r)}\int_{\mathbb{R}} \Big(z^2\land r^2\Big)\nu(\ud z) \to \infty \qquad \textrm{as $r$ goes to $0$,}
\end{equation} 
turns out to be a necessary and sufficient  for the  transition density  being in  $\mathcal{C}^\infty_b$ (see Theorem 1 in \cite{BodK}). We also point out that the previous condition differs from the so-called {\it Kallenberg's condition}, i.e.
\begin{equation}\label{kalle}
-\frac{1}{r^2\ln(r)}\int_{-r}^r z^2\nu(\ud z) \to \infty \qquad \textrm{as $r$ goes to $0$,}
\end{equation} 
which  is a necessary condition for the density of the L\'evy process $\xi$ being in $\mathcal{C}^\infty_b$ (see Section 5 in Kallenberg \cite{kal}). Indeed, the L\'evy measure $\nu =\sum_{n\ge 1} n\delta_{1/n!}$ satisfies  \eqref{BodK1cond} but not {\it Kallenberg's condition}. Actually for such example, we have the following unexpected behaviour which is  that  the distribution of the L\'evy process $\xi$ is singular but the transition density of its associated OUL process is in $\mathcal{C}^\infty_b$ (see Example 1 in \cite{BodK1}). In other words, the  drift given by the dynamics \eqref{sdem} may provide enough regularity to the transitions even if the distribution of the noise is singular.

It seems that we cannot expect a weaker condition for the  transition density being in   $\mathcal{C}^\infty_b$ than \eqref{BodKcond}, since according to Bodnarchuk and Kulyk \cite{BodK1}, the following condition 
\begin{equation}\label{borna}
-\frac{1}{r^2\ln(r)}\sup_{\ell \in \mathbb{S}^{d-1}}\int_{\mathbb{R}^d} \Big(\langle z, \ell\rangle^2\land r^2\Big)\nu(\ud z)\to \infty
\qquad \textrm{as $r$ goes to $0$,}
\end{equation}
is necessary for the existence of a  bounded continuous density.

Before we introduce the invariant distribution of the process $X^{(\epsilon)}$, we recall the notion of  {\it self-decomposability operator} 
of a distribution on $\mathbb{R}^d$. Let  $Q\in \mathcal{M}^+(d)$, then an  infinitely divisible distribution $\mu$ on $\mathbb{R}^d$ is called {\it $Q$-self-decomposable} if there exists a probability distribution $\eta_{t,Q}$ such that, for each $t\ge 0$,
\[
\widehat{\mu}(\lambda):=\int_{\mathbb{R}^d} e^{i\langle \lambda, z\rangle} \mu(\ud z)=\widehat{\mu}\left(e^{-tQ^T} \lambda\right)\widehat{\eta}_{t,Q}(\lambda), \qquad \lambda \in \mathbb{R}^d,
\]
where $\widehat{\eta}_{t,Q}$ denotes the characteristic function or Fourier transform of $\eta_{t,Q}$.
An infinitely divisible distribution $\mu$ on $\mathbb{R}^d$ which is $Q$-self-decomposable for some $Q\in \mathcal{M}^+(d)$ is called {\it operator self-decomposable}. If $d=1$, then the operator self-decomposability property reduces to self-decomposability. It is important to note that the support of any $Q$-self-decomposable distribution is unbounded except for delta distributions (see for instance Corollary 24.4 in Sato \cite{Sa}).

In the sequel, we  assume that the L\'evy processes $\xi$ satisfies the following log-moment condition
\begin{equation}\label{cond1}
\mathbb{E}\Big[\log(1\lor |\xi_1|)\Big]<\infty,
\end{equation}
where $a\lor b$ denotes the maximum between the numbers $a$ and $b$, which  is equivalent to
\[
\int_{D_1^c} \log|x|\nu(\ud x)<\infty,
\]
where $D_{r}:=\{z\in \mathbb{R}^d: |z|\le r\}$,  for $r\ge 0$ , and $D_{r}^{c}:=\mathbb{R}^d\setminus D_{r}$ (see  Theorem 25.3 in Sato \cite{Sa}). 
The log-moment condition \eqref{cond1} is  necessary and sufficient  for the existence of  a stationary distribution for the process $X^{(\epsilon)}$, here denoted by  $\mu^{(\epsilon)}$, and it satisfies
\[
\int_{\mathbb{R}^d} e^{i\langle \lambda, x\rangle }\mu^{(\epsilon)} (\ud x)= \exp\left\{\int_0^\infty \psi\Big(\sqrt{\epsilon}e^{-sQ^T} \lambda\Big) \ud s\right\}, \qquad \lambda \in \mathbb{R}^d,
\]
see for instance Theorems 4.1 and 4.2 in \cite{SaYa} (or Theorem 17.5 in \cite{Sa} for the case $Q=q {\rm I}_{\textrm{d}}$ for $q>0$ and where ${\rm I}_{\textrm{d}}$ denotes the identity matrix). Moreover, the distribution $\mu^{(\epsilon)}$ is  $Q$-self-decomposable which is determined by a triple $(a^{(\epsilon)}_\infty, \sigma^{(\epsilon)}_\infty, \nu^{(\epsilon)}_\infty)$ given by
\[
a^{(\epsilon)}_\infty= \sqrt{\epsilon}Q^{-1}a+\sqrt{\epsilon}\int_{\mathbb{R}^d}\nu(\ud y) \int_0^\infty e^{-sQ} y\left(\mathbf{1}_{\{|\sqrt{\epsilon}e^{-sQ} y|\le1\}}-\mathbf{1}_{\{|y|\le1 \}}\right)  \ud s,
\]
and
\[
\qquad  \Sigma^{(\epsilon)}_\infty=\epsilon \int_0^\infty e^{-sQ }\Sigma e^{-sQ^T } \ud s,\qquad
 \nu^{(\epsilon)}_\infty(B)=\int_0^\infty \nu(\sqrt{\epsilon}e^{sQ} B)\ud s \qquad \textrm{for} \quad B\in \mathcal{B}(\mathbb{R}^d).\]
In fact, Theorem 4.1 in \cite{SaYa} determines  the class of all of   $Q$-self-decomposable distributions as the class of all possible invariant distributions of OUL. According to Yamazato \cite{Ya} if $\mu^{(\epsilon)}$ is  non-degenerate  then $\mu^{(\epsilon)}$ is absolutely continuous with respect to the Lebesgue measure on $\mathbb{R}^d$. However, 
not so much  information about the regularity of the density  can be found in the literature, up to our knowledge.  In the one dimensional case, the distribution $\mu^{(\epsilon)}$ is self-decomposable and, if it is non-degenerate, then it is absolutely continuous with respect to the Lebesgue measure and its density is increasing on $(-\infty, \wp)$ and decreasing on $(\wp,\infty)$, where $\wp\in \mathbb{R}$ is known as the mode (see for instance Theorem 53.1 in \cite{Sa}).

Computing explicitly the density of the  invariant distribution $\mu^{(\epsilon)}$  is rather complicated even in the one dimensional case but in some specific examples we can say something about it. For instance, if the L\'evy process $\xi$ admits  a continuous density and $\mu^{(\epsilon)}$ has a smooth density then the latter can be determined explicitly, see Remark 2.3 in \cite{MA}. In the case when $\xi$ is a subordinator with finite jump measure (i.e.  $\nu(0,\infty)<\infty$), the asymptotic behaviour of its density   near $0$ can be  established (see for instance Theorem 53.6 in \cite{Sa}).

\subsection{Main results}

Recall that $d^{(\epsilon)}(t)$ denotes  the total variation distance between the distribution of $X^{(\epsilon)}_t$  and its invariant distribution $\mu^{(\epsilon)}$, that is to say
\[
d^{(\epsilon)}(t)=\norm{X^{(\epsilon)}_t-X^{(\epsilon)}_\infty} \qquad \textrm{for } \quad t\ge 0,
\]
where $X_\infty^{(\epsilon)}$ denotes the limit distribution of $X^{(\epsilon)}$ whose  law is given by $\mu^{(\epsilon)}$.

We also introduce the L\'evy process $\xi^\natural$ as follows $\xi^\natural_t:=\xi_t-at$, for $t\ge 0$,  and its associated 
exponential functional $I^\natural=(I^\natural_t, t\ge 0)$ which is defined by
\[
I^\natural_t=\int_0^t e^{-(t-s)Q}\ud \xi^\natural_s \qquad \textrm{for} \quad t\ge 0.
\]
We observe that  its limiting distribution is well-defined under the log-moment condition \eqref{cond1}, here denoted by $I^\natural_\infty$.
We also denote by $\mu^\natural_t$ the distribution of $I^\natural_t$, for $t\ge 0$, and $\mu^\natural_\infty$ the distribution of $I^\natural_\infty$. For the sequel, we assume that for any $t>0$,
 \[
 \hspace{-1cm}\mathrm{\bf (H)}\hspace{2cm}\widehat{\mu^\natural}_t(\cdot) \quad\textrm{ is integrable}\quad \textrm{and}\quad\lim_{R \to \infty} \sup_{s> t_0(R)}\int_{D^c_R}|\widehat{\mu^\natural}_s(\lambda)| \ud \lambda =0, 
 \]
 where    $\widehat{\mu^\natural}_t$ denotes the characteristic function of  $\mu^\natural_t$,
  $D_{R}^{c}:=\mathbb{R}^d\setminus D_{R}$ and $t_0(R)$ is  positive and goes to $\infty$ as $R$ increases.  
The integrability  condition on  $\widehat{\mu^\natural}_t$ implies that $I^\natural_t$, for $t>0$, has a continuous density $f_t(x)$ that goes to $0$ as $|x|$ goes to $\infty$ (see for instance  Proposition 28.1 in \cite{Sa}). If the log-moment condition \eqref{cond1} also holds  then $I^\natural_\infty$ also has a continuous densities  $f_\infty(x)$ that goes to $0$, as $|x|$ goes to $\infty$. It is not so difficult to deduce that the same property holds for the density of $X^{(1)}_t$ and its invariant distribution $\mu^{(1)}$. Indeed, the latter holds since 
$I^\natural_t$ and $X^{(1)}_t$  differ only in the drift term and 
\[
 \left|\widehat{\mu^{(1)}}(\lambda)\right|\le \left|\widehat{\mu^\natural}_t(\lambda)\right| \qquad \textrm{for any}\quad t>0,
\]
where $|z|$ denotes the modulus of $z\in \mathbb{C}$. For simplicity, we use the same notation  between the Euclidian norm and the modulus for  the complex numbers. 
The second condition in {\bf (H)} guarantees the convergence, under the total variation distance, of $I^\natural_t$ towards $I^\natural_\infty$, as $t$ increases.

Recall that  the matrix $Q$ belongs to $\mathcal{M}^+(d)$, i.e. it has eigenvalues with positive real part.   The following lemma, whose general proof can be found in Barrera and Jara \cite{BJ1} (see Lemma B.1) and we state it here for the sake of completeness, provides the asymptotic behaviour of $e^{-tQ}$, for large $t$. Such asymptotic behaviour is necessary for determining the cut-off times.

\begin{lemma}\label{jara}
Let $Q\in \mathcal{M}^+(d)$.
For any $x_0 \in \mathbb{R}^d\setminus\{0\}$ there exist $\gamma:=\gamma(x_0)>0$,
$\ell:=\ell(x_0) , m:=m(x_0)  \in  \{1,\ldots, d\}$, $\theta_1:=\theta_1(x_0),\dots,\theta_{m}:=\theta_m(x_0) \in [0,2\pi)$ and $v_1:=v_1(x_0),\dots,v_{m}:=v_m(x_0) \in  \mathbb{C}^d$ linearly independent such that
\[
\lim_{t \to \infty} \left |\frac{e^{\gamma t}}{t^{\ell-1}} e^{- tQ}x_0 - \sum_{k=1}^{m} e^{i  t\theta_k} v_k\right | =0.
\]
Moreover,
\[
0<\liminf_{t\rightarrow \infty}\left|\sum_{k=1}^{m} e^{i  t\theta_k} v_k\right|
\leq 
\limsup_{t\rightarrow \infty}\left|\sum_{k=1}^{m} e^{i  t\theta_k} v_k\right|\leq 
\sum_{k=1}^{m}  |v_k|.
\]
\end{lemma}
The numbers $\{\lambda \pm \theta_k, k=1,\ldots, m\}$ are eigenvalues of the matrix $Q$ and the vectors $\{v_k, k=1,\ldots, m\}$ are elements of the Jordan decomposition of $Q$. For a better understanding of the asymptotic behaviour   of $e^{- tQ}$ and its role,  we provide the proof of  Lemma \ref{jara} in the case when all the eigenvalues of $Q$ are positive real numbers, see Proposition \ref{lyapunov} in the Appendix.

Now, we present the main result of this paper using the same notation as in the previous lemma.

\begin{theorem}\label{teo2}
Let $Q\in \mathcal{M}^+(d)$ and  $x_0\in \mathbb{R}^d\setminus\{0\}$. Assume that  the log-moment condition \eqref{cond1} and  {\bf (H)} hold;  and  take $0<\epsilon\ll 1$ such that 
\begin{equation}\label{twcut-off}
t_\epsilon:=(2\gamma)^{-1}\ln\left(\nicefrac{1}{\epsilon}\right)+
\frac{\ell-1}{\gamma}
\ln\left(\ln\left(\nicefrac{1}{\epsilon}\right)\right)>0\quad
\mathrm{  and  }\quad
w_\epsilon:=\gamma^{-1}+o_\epsilon(1)>0,
\end{equation} 
where $\lim\limits_{\epsilon\rightarrow 0}o_\epsilon(1)=0$, then
\begin{itemize}
\item[i)]  if $\theta_k=0$ for every $k\in \{1,\ldots,m\}$, 
the  family of  OUL processes  $(X^{(\epsilon)}, \epsilon>0)$ 
has profile cut-off  with cut-off time $t_\epsilon$, time window $w_\epsilon$ and profile function $G_{x_0}:\mathbb{R}\rightarrow [0,1]$ given by  
\[
\begin{split}
G_{x_0}(c):=\lim\limits_{\epsilon\rightarrow 0}d^{(\epsilon)}(t_\epsilon+cw_\epsilon)
=
\norm{\left(\left(2\gamma\right)^{1-\ell}e^{-c}v(x_0)+I^{\natural}_\infty\right)-I^\natural_\infty},
\end{split}
\]
for every $c\in \mathbb{R}$, where $v(x_0):=\sum_{k=1}^{m} v_k$.  In addition, the profile function $G_{x_0}$ satisfies
\[
\lim\limits_{c \rightarrow -\infty}{G_{x_0}(c)}=1\qquad \textrm{and} \qquad \lim\limits_{c \rightarrow \infty}{G_{x_0}(c)}=0.
\]
\item[ii)] 
if $\theta_k\not=0$ for some $k\in \{1,\ldots,m\}$,  the  family of  OUL processes  $(X^{(\epsilon)}, \epsilon>0)$ 
has window cut-off with cut-off time $t_\epsilon$ and time window $w_\epsilon$.
\end{itemize}
\end{theorem}

Recall that the cut-off times $t_\epsilon$ and the windows cut-off $w_\epsilon$ depend on the initial condition. We omit such dependence for simplicity of exposition. It is also important to note that when  the initial condition $x_0=0$, we have $X^{(\epsilon)}_t=\sqrt{\epsilon}X^{(1)}_t$ for  $t\geq 0$, and $X^{(\epsilon)}_\infty=\sqrt{\epsilon}X^{(1)}_\infty$.
From Lemma \ref{sca} part (ii) (see Appendix), we obtain
\[
\norm{X^{(\epsilon)}_t-X^{(\epsilon)}_\infty}=
\norm{X^{(1)}_t-X^{(1)}_\infty} \qquad \textrm{ for  }\quad t\geq 0,
\]
that is to say  the  cut-off phenomenon does not occur. Moreover, Lemma \ref{jara} does not hold if $x_0=0$. For that reason, we  assume  that $x_0\not=0$. 

We also observe  that the profile function $G_{x_0}$, when it exists,  depends on the initial condition $x_0$ and is given in terms of 
the total variation distance between two $Q$-self-decomposable distributions  that do not depend on the drift term of the underlying L\'evy process since such part is deterministic. The previous observation is very interesting since most of the examples that appears in the literature (mainly for Markov chains such as the random walk on the hypercube) that exhibits profile cut-off are given in terms of the Gauss error function. Up to our knowledge, the only examples  that exhibit profile cut-off and  do not fulfill  the previous property are the  top-to-random
shuffle and the transposition shuffle for which the important statistic is
the number of fixed points which behaves like a Poisson random variable, see Lacoin \cite{LAC} and the references therein.

We also point-out that  part (i) of our main result includes the case when $Q$ has real eigenvalues which is easier to understand. Indeed,  denote by   $\gamma_1\le \gamma_2 \le \cdots \le\gamma_d$ for the eigenvalues associated to $Q$. If $Q$ is a symmetric matrix (or that the process \eqref{sdem} is reversible) with different eigenvalues then we can characterise explicitly $\gamma, \ell$ and $v$ in a very simple way using  the celebrated Spectral Theorem. In other words,
there exists an orthonormal basis $\{v_1,v_2,\ldots, v_d\}$ of $\mathbb{R}^d$ for which
\[
e^{-tQ}x_0=\sum\limits_{j=1}^{d}{e^{-\gamma_j t}\langle x_0,v_j\rangle v_j} \quad \textrm{for}\quad t\geq 0.
\]
Define
$\tau(x_0):=\min\{k\in\{1,\ldots,d\}:\langle x_0,v_k\rangle \not=0\}$ 
 and take $\gamma=\gamma_{\tau(x_0)}$. Then
\[
e^{\gamma t}e^{-tQ }x_0=\langle x_0,v_{\tau(x_0)}\rangle v_{\tau(x_0)}+\sum\limits_{j=\tau(x_0)+1}^{d}{e^{-(\gamma_j-\gamma) t}\langle x_0,v_j\rangle v_j} \quad \textrm{for}\quad t\geq 0,
\]
with the understanding  that if $\tau(x_0)=d$, then the second term of the right-hand side of the above identity equals 0.
Consequently,
\[
\lim\limits_{t\rightarrow \infty}e^{\gamma t}e^{-tQ }x_0=\langle x_0,v_{\tau(x_0)}\rangle v_{\tau(x_0)}\in \mathbb{R}^d\setminus\{{0}\}.
\]
If $Q$ is still symmetric but some of the eigenvalues may repeat, the values of $\gamma, \ell$ and $v$ can also be determined using the matrix  diagonalisation method. For more details see Proposition \ref{lyapunov}, part (i) in the Appendix.

In the particular case when  $d=1$ and    conditions \eqref{cond1}  and  {\bf (H)} are satisfied, we always have profile cut-off for the OUL process.
Moreover, we have that  $\mu^{(1)}$ and $I^\natural_\infty$ are self-decomposable (see Theorem $17.5$ in \cite{Sa}).

For a general $Q\in \mathcal{M}^{+}(d)$ and since part (ii) of Theorem \ref{teo2}  only implies   window cut-off for the  family of  OUL processes  $(X^{(\epsilon)}, \epsilon>0)$, then a natural question arises: \textit{Are there cases where profile cut-off exist for general $Q$?}  There is an affirmative  answer to  this question  which depends on   the following invariance property,
\begin{equation}
\label{invarianza1}
\norm{(v_1+I^{\natural}_\infty)-I^{\natural}_\infty}=\norm{(v_2+I^{\natural}_\infty)-I^{\natural}_\infty}
\end{equation}
for any $v_1,v_2\in \mathbb{R}^d$ such that $|v_1|=|v_2|$. 

\begin{corollary}\label{pt}
Let $Q\in \mathcal{M}^+(d)$ and  $x_0\in \mathbb{R}^d\setminus\{0\}$. Assume that  the log-moment condition \eqref{cond1},  {\bf (H)}  and  the invariance property \eqref{invarianza1} hold  and we take $0<\epsilon\ll 1$ such that 
$t_\epsilon$ and $w_\epsilon$ are defined as in \eqref{twcut-off}. If 
\[\left|\sum\limits_{k=1}^{m}e^{i\theta_k t}v_k\right|  \textrm{ is constant},\]
then there is profile cut-off for the  family of  OUL processes  $(X^{(\epsilon)}, \epsilon>0)$.\end{corollary}
We point out that the invariance property \eqref{invarianza1} is satisfied when the limiting distribution $\mu^\natural_\infty$ is isotropic, i.e. that it is invariant under orthogonal transformations. Examples of isotropic distributions which are also self-decomposable are    the standard Gaussian and isotropic stable distributions.

The strategy to deduce our main result  is as follows. We first  introduce an auxiliary metric which approximates  the metric $d^{(\epsilon)}$ quite well. The idea of using this auxiliary metric comes from the Brownian setting where  the distributions  of $X_t^{(\epsilon)}$ and $X_\infty^{(\epsilon)}$ are known and everything is much easier to control 
thanks to the good estimates that one can deduce  under  the total variation distance. More precisely, both distributions are Gaussian  and the convergence of such auxiliary metric depend on the speed of convergence of the mean of $X_t^{(\epsilon)}$ towards the mean of $X_\infty^{(\epsilon)}$, under the total variation distance. In the L\'evy case such estimates are not available (even in the stable case)  and henceforth a new approach is needed. Indeed,  we use Scheff\'e's Lemma together with Lemma \ref{jara} to study the behaviour of the auxiliary metric. Here is where the cut-off times appear by analysing the drift term $e^{-Qt}x_0$ and using the scaling parameter $\sqrt{\epsilon}$ that appears on the noise. Then to control the error term, we use the Fourier inversion of the characteristic functions of $X_t^{(\epsilon)}$ and  its invariant measure,  together with 
Masuda-type estimates as those used in \cite{MA}.

The remainder of this paper is organised as follows.  In section \ref{fresults}, the statements of  the cut-off phenomenon of the superposition and average processes are established.
Section \ref{smoothness} provides few examples where the assumption {\bf (H)} is fulfilled. 
Section \ref{proofs} is devoted to  the proofs of the results of this paper.
Finally, in the  Appendix  some tools that we omit along this paper are established.

\section{Further results}\label{fresults}
\subsection{Superposition process}
A simple and nice way to model  observational processes that show significant dependence over long time periods is by means of superposition of independent processes with short-range dependence. In this setting, superposition of independent OU type processes have provided flexible and analytically tractable parametric models, see for instance Barndorff-Nielsen \cite{BN} and the references therein.

On the other hand the superposition of independent OU type processes, that we  consider here, can be associated with  an example of a cylindrical OU process which is defined in terms of an infinite-dimensional Langevin equation, see Section  7 in Applebaum \cite{Appl} for further details. As it is noted in \cite{Appl}, infinite-dimensional processes arise naturally in mathematical modelling through 
noise that is described as ``superposition" of independent real-valued L\'evy processes.

In the sequel, we take the parameter $\epsilon>0$ and introduce $(\xi^{(j)}, j\ge 1)$ a sequence of  independent real-valued L\'evy processes which are not necessarily equally distributed. For each $j\ge 1$, we assume that $\xi^{(j)}$ has characteristics  $(a_j,\sigma_j,\pi_j)$ satisfying $a_j\in \mathbb{R}$, $\sigma_j\ge 0$ and 
\[
\int_{\mathbb{R}}{(1\land z^2 )\pi_j(\ud z)}<\infty.
\]

Similarly as before, for each $j\ge 1$, we also consider their associated OUL processes $X^{(\epsilon,j)}$ which are defined by
\begin{equation*}\label{sde22}
\left\{
\begin{array}{r@{\;=\;}l}
\ud X^{(\epsilon,j)}_t & -\gamma_j X^{(\epsilon,j)}_t\ud t+\sqrt{\epsilon}\ud \xi^{(j)}_t \qquad \textrm{for} \qquad t\geq 0, \\
X^{(\epsilon,j)}_0 & x_j,
\end{array}
\right.
\end{equation*} 
where $\gamma_j>0$.
Let  $m=(m_j, j\geq 1)$ be a sequence of real positive numbers 
such that
$\sum_{j\ge 1}{m_j}=1$ and define the {\em{superposition process}} $\chi^{(\epsilon)}:=(\chi^{(\epsilon)}_t: t\geq 0)$ as follows
\begin{equation}
\label{suppp}
\chi^{(\epsilon)}_t:= \sum\limits_{j=1}^{\infty}{m_jX^{(\epsilon,j)}_t} \qquad \textrm{for}\quad t\geq 0.
\end{equation}

We now introduce a series of assumptions that guarantee  that the superposition process $\chi^{(\epsilon)}$ is well-defined.
We first assume that the {\em{initial configuration}} $x:=(x_j:j\geq 1)$ is $m$-integrable, that is to say, 
\begin{equation}\label{C2}
\sum_{j=1}^{\infty}{m_j|x_j|}<\infty.
\end{equation}
The next conditions guarantee  that the drift and Gaussian terms of \eqref{suppp} are well-defined, 
\begin{equation}\label{C3}
\sum\limits_{j=1}^{\infty}\frac{m_j|a_j|}{\gamma_j}<
\infty \qquad \textrm{and}\qquad \sum\limits_{j=1}^{\infty}
\frac{m^2_j \sigma_j}{\gamma_j}<\infty.
\end{equation}
For the jump structure of the process, some additional  conditions are needed.  Let us assume that the L\'evy measures $(\pi_j:j\geq 1)$ satisfies 
\begin{equation}\label{nearzero}
\sum_{j=1}^{\infty}{\frac{m^2_j}{\gamma_j}\int_{(-1,1)}z^2\pi_j(\ud z)}< \infty, \qquad \sum_{j=1}^{\infty}\frac{m_j}{\gamma_j}\pi_j ([-1,1]^c)<\infty,
\end{equation}
and
\begin{equation}\label{nearzero1}
\sum\limits_{j=1}^{\infty}{\frac{1}{\gamma_j}\int_{[-1,1]^c}\log |z|\pi_j(\ud z)}<\infty.
\end{equation}

\begin{lemma}\label{opo}
Let $x:=(x_j:j\geq 1)$ be an initial configuration satisfying \eqref{C2} and take $\epsilon>0$. We also  assume that conditions  \eqref{C3},  \eqref{nearzero} and \eqref{nearzero1} are satisfied. 
Then  the superposition process $(\chi^{(\epsilon)}_t, t\ge0)$ is well-defined and  has a limit distribution $\mu^{(\epsilon, m)}$ which is independent of the initial configuration $x$. Moreover, $\mu^{(\epsilon, m)}$ is self-decomposable with characteristics
\[
a^{(\epsilon,m)}_\infty= \sqrt{\epsilon}\sum\limits_{j=1}^{\infty}\frac{m_ja_j}{\gamma_j}+\sqrt{\epsilon}\sum\limits_{j=1}^{\infty} m_j\int\limits_{\mathbb{R}}\pi_j(\ud x)\int\limits_{0}^{\infty}e^{-\gamma_j s}x\left(\mathbf{1}_{\{|e^{-\gamma_j s }x|\le 1\}}-\mathbf{1}_{\{|x|\le 1\}}\right)\ud s,
\]
and
\[
\qquad  \sigma^{(\epsilon, m)}_\infty=\sum\limits_{j=1}^{\infty}
\frac{m^2_j \sigma_j}{2\gamma_j},\qquad
 \pi^{(\epsilon, m)}_\infty(B)=\sum_{j=1}^\infty \int_0^\infty \pi_j(\sqrt{\epsilon}m^{-1}_je^{s\gamma_j} B)\ud s \qquad \textrm{for} \quad B\in \mathcal{B}(\mathbb{R}).\]
\end{lemma}

Since $\mu^{(\epsilon, m)}$ is a self-decomposable random variable on $\mathbb{R}$ then it is degenerated or absolutely continuous with respect to the Lebesgue measure. As we mentioned before, from Theorem 53.1 in \cite{Sa} 
its density, when it exists, is unimodal. From  Proposition 2.4 in \cite{MA} or Proposition 24.19 in \cite{Sa}, it follows that
$\mu^{(\epsilon, m)}$ is non-degenerate if and only if the limit distribution of the process $X^{(\epsilon,j)}$ is non-degenerate for some $j\geq 1$.

For the main result in this section, we assume that the friction coefficients are uniformly bounded away from $0$. In other words,
we assume uniform {\em{coercivity}} as follows
\begin{equation}\label{C1}
\textrm{there exists } \gamma>0 \textrm{ such that } \gamma_j\geq \gamma \text{ for any } j\geq 1.
\end{equation}
We also consider  the sequence of L\'evy process $(\xi^{(\natural,j)}, j\ge 1)$ which is defined as follows: for each $j\ge 1$, $\xi^{(\natural,j)}$ has characteristics  $(0,\sigma_j,\pi_j)$ satisfying \eqref{C3}, \eqref{nearzero} and \eqref{nearzero1}. For the sequence $(\xi^{(\natural,j)}, j\ge 1)$, we also introduce its associated functional $I^{\natural, m}=(I^{\natural, m}_t, t\ge 0)$ which is defined by
\[
I^{\natural, m}_t=\sum_{j=1}^\infty m_j\int_0^t e^{-\gamma_j(t-s)}\ud \xi^{(\natural, j)}_s \qquad \textrm{for} \quad t\ge 0.
\]
We observe that  its limiting distribution $I^{\natural, m}_\infty$ is well-defined under  conditions \eqref{C3}, \eqref{nearzero} and \eqref{nearzero1}.
 

For every $\epsilon>0$ and $t\geq 0$, define 
\[
d^{(\epsilon, m)}(t):=\norm{\chi_t^{(\epsilon)}-\chi^{(\epsilon, m)}_\infty},
\] 
where $\chi^{(\epsilon, m)}_\infty$ denotes the limiting distribution of $\chi^{(\epsilon)}$ whose law is given by $\mu^{(\epsilon, m)}$.
\begin{theorem}\label{supou} Let  $x:=(x_j:j\geq 1)$ be a initial configuration satisfying \eqref{C2}  and 
assume that conditions  \eqref{C3}, \eqref{nearzero} and  \eqref{nearzero1} hold. 
We also assume that  there exist $j\ge 1$ such that  for any $t>0$,  the distribution of 
\[
I^{(\natural,  j)}_t:=\int_0^t e^{-\gamma_j(t-s)}\ud \xi^{(\natural, j)}_s,
\]
satisfies hypothesis {\bf (H)}, 
the coercivity condition \eqref{C1} is fulfilled, 
$J:=\{j\geq 1: \gamma_j=\inf\limits_{k\in \mathbb{N}}{\gamma_k}\}\not= \emptyset$ and $\sum\limits_{j\in J}{m_jx_j}\not=0$. Then, the family $(\chi^{(\epsilon)}, \epsilon>0)$
possesses  profile cut-off under the total variation distance when $\epsilon$ decreases to $0$, with cut-off time and  window cut-off given by 
\begin{eqnarray*}
t_\epsilon:=\frac{1}{2\hat{\gamma}}\ln\left(\nicefrac{1}{\epsilon}\right) \qquad \textrm{  and  }\qquad
w_\epsilon:=\frac{1}{\hat{\gamma}}+o(1),
\end{eqnarray*}
where $\hat{\gamma}:=\inf\limits_{k\in \mathbb{N}}{\gamma_k}$. The profile function $G_{x,m}:\mathbb{R}\rightarrow [0,1]$ is given by
\begin{eqnarray*}
G_{x,m}(c):=\lim\limits_{\epsilon\rightarrow 0}{d^{(\epsilon, m)}\left(t_\epsilon+cw_\epsilon\right)}=\norm{\left(e^{-c}\sum\limits_{j\in J}{m_jx_{j}}+I^{\natural, m}_\infty\right)-I^{\natural, m}_\infty},
\end{eqnarray*}
for any $c\in \mathbb{R}$, and satisfies
\[
\lim\limits_{c \rightarrow -\infty}{G_{x.m}(c)}=1\qquad \textrm{and} \qquad \lim\limits_{c \rightarrow \infty}{G_{x,m}(c)}=0.
\]
\end{theorem}

\subsection{Average  process}
Finally, we  study  the cut-off phenomenon  for the average  of OUL when the driving process is a stable L\'evy process.  In the diffusive case, Lachaud \cite{BL} observed that the average process  satisfies window cut-off with the same cut-off  and window times as the sample of OU processes. The previous observation is quite  surprising since the sample process comprises a huge amount  number of processes. As we will see below, the  average process of OU  not only  possesses cut-off  and window cut-off   but also has  profile cut-off.

Let  us consider the sequence  $(\epsilon_n, n\ge 1)$ of strictly positive real numbers    converging to $0$ accordingly as $n$ increases.
In what follows,  we assume that the  process  $\xi$ is  a real-valued  stable L\'evy process with a linear drift $a\in \mathbb{R}$,  that is to say
 its characteristic exponent  $\psi_\alpha$ is given by
\begin{equation}\label{stable}
\psi_\alpha(z)=iz a- c|z|^{\alpha}\left(1-i\beta\tan(\pi\alpha/2)\mathrm{sgn}(z)\right) \qquad \textrm{for}\quad z\in \mathbb{R},
\end{equation}
where $\alpha\in (0,1)\cup(1,2]$, $c>0$ and $\beta\in[-1,1]$ or $\alpha=1$, $\beta=0$ and we understand
$\beta\tan(\pi\alpha/2)=0.$

We also consider the sequence of OUL  processes $((X^{(\epsilon_n)}_t: t\geq 0),  n\ge 1)$ such that for each $n\ge 1$, $X^{(\epsilon_n)}$ is defined as the unique strong solution of \eqref{sdem} with $\gamma:=Q>0$  and initial condition $X^{(\epsilon_n)}_0=x_0\neq 0$. Let $\big(\big(X^{(\epsilon_n),1}_t,\cdots,X^{(\epsilon_n),n}_t\big), t\geq 0 \big)$ be a sample of $n$ independent copies of  $X^{(\epsilon_n)}$.  For simplicity on exposition,  we denote by $(\xi^{(j)}, 1\le j\le  n)$ for the sequence of independent copies of the stable L\'evy process $\xi$ which drives the  above  sample of OUL.

For each $n\ge 1$, we define the  average process $A^{(n)}:=(A^{(n)}_t, t\geq 0)$ as follows
\begin{eqnarray}\label{ap}
A^{(n)}_t&:=&\frac{1}{n}\sum\limits_{i=1}^{n}{X^{(\epsilon_n),i}_t}\qquad \textrm{for} \qquad t\geq  0.
\end{eqnarray}
It is not so difficult to deduce that the uniform average process $A^{(n)}$ satisfies the following SDE
\[
\ud A^{(n)}_t =-\gamma A^{(n)}_t\ud t+\sqrt{\epsilon_n} \ud L^{(n)}_t \qquad \textrm{for} \qquad t\geq 0,
\]
where $L^{(n)}:=(L^{(n)}_t,t\ge 0)$ is a stable L\'evy process with drift such that 
\[L^{(n)}_t:=
\frac{1}{n}\sum_{i=1}^n\xi^{(i)}_t\qquad \textrm{for} \qquad t\geq 0.
\]
 It is straightforward to deduce that the characteristic exponent of $L^{(n)}$ is given by 
 \[
 \psi^{L^{(n)}}_\alpha(z)=iz a- cn^{1-\alpha}|z|^{\alpha}\left(1-i\beta\tan(\pi\alpha/2)\mathrm{sgn}(z)\right) \qquad \textrm{for}\quad z\in \mathbb{R}.
 \]
Since the stable  L\'evy process $\xi$  satisfies the log-moment condition \eqref{cond1} for  $\alpha\in(0,2)$, the  average process $A^{(n)}$ has a limiting distribution, that we denote by $A^{(n)}_\infty$. On the other hand, it is well known that the limiting distribution also exists when $\xi$ is a Brownian motion with drift, i.e. when $\alpha=2$. In any case, the characteristic exponent  of $A^{(n)}_\infty$ is as follows
\[
 \psi^{A^{(n)}_\infty}_\alpha(z)=iz a \frac{\sqrt{\epsilon_n}}{\gamma}-\frac{cn^{1-\alpha}\epsilon_n^{\alpha/2}}{\alpha\gamma}|z|^{\alpha}\left(1-i\beta\tan(\pi\alpha/2)\mathrm{sgn}(z)\right) \qquad \textrm{for}\quad z\in \mathbb{R}.
 \]
For each $n\ge 1,$ we also define the total variation distance between $A^{(n)}_t$ and its limiting distribution by 
\[ 
 d^{(n)}(t):=\norm{A^{(n)}_t-A^{(n)}_\infty} \qquad \textrm{for}\quad t\ge 0.
 \]

\begin{theorem}\label{ave} For $x_0\ne 0$, the family of processes $(A^{(n)}, n\ge 1)$ possesses  profile cut-off under the total variation distance, when $n$ goes to $\infty$, with 
 cut-off time and window cut-off  given by 
\begin{eqnarray*}
t_n:=\frac{1}{2\gamma}\ln\left(\frac{n^{2-2/\alpha}}{\epsilon_n}\right),\qquad
w_n:=\frac{1}{\gamma}+o_n(1).
\end{eqnarray*}
The  profile  function $G:\mathbb{R}\rightarrow [0,1]$ is given by
\begin{eqnarray*}
G_{x_0}(c):=\lim\limits_{n\rightarrow \infty}{d^{(n)}\left(t_n+cw_n\right)}=\norm{(e^{-c}x_0+\mathcal{S}_\alpha)-\mathcal{S}_\alpha },
\end{eqnarray*}
for any $b\in \mathbb{R}$ and where $\mathcal{S}_\alpha$ is a strictly stable distribution, i.e. it characteristic exponent is given by $\psi_\alpha$ with $a=0$. Moreover, it  satisfies
\[
\lim\limits_{c \rightarrow -\infty}{G_{x_0}(c)}=1\qquad \textrm{and} \qquad \lim\limits_{c \rightarrow \infty}{G_{x_0}(c)}=0.
\]
\end{theorem}
It is important to note that the assumption that $\xi$ is a stable L\'evy process with drift is crucial in our arguments. Indeed,  the dimension of the sampling and the cut-off times $t_n$ in the distance $d^{(n)}$ are very strong related that without the scaling property seems to be very difficult to deduce any limiting behaviour of  $d^{(n)}(t_n)$. To be more precise the weak limit of 
\[
\int_0^{t_n} e^{-\gamma(t_n-s)}\ud L^{(n)}_{s},
\]
under the total variation distance needs to be well-understood.

\section{Smoothness}\label{smoothness}
In this section, we provide a few examples where condition {\bf (H)} is satisfied. Moreover, in all examples presented below the marginal distribution of the OUL process $X^{(\epsilon)}$ (and the superposition process $\chi^{(\epsilon)}$), for any $\epsilon>0$, has a density  in $\mathcal{C}_b$ or $\mathcal{C}_b^\infty$. Implicitly the invariant distribution  $\mu^{(\epsilon)}$ (similarly for $\mu^{(\epsilon, m)}$) and the random variable $I^\natural_\infty$ (similarly for  $I^{\natural, m}_\infty$) have densities belonging to $\mathcal{C}_b$ or $\mathcal{C}_b^\infty$.

For simplicity, we use the notation  $\Re(z)$ and $\Im(z)$ for the real  and imaginary part of any complex number $z$.

{\bf 1.} The first example that we consider here is the case when $ \Sigma$ is positive definite, i.e.  the L\'evy process $\xi$ has presence of a $d$-dimensional Brownian motion. In other words, the matrix $ \Sigma$ has full rank and implicity
 for any $t>0$, $X_t^{(1)}$ and $I_t^{\natural}$  have densities belonging to $\mathcal{C}_b^\infty$ (see Masuda \cite{MA}). Both distributions possess
  a Gaussian  component  which are described by the covariance matrix 
\[
\Sigma_t= \int_0^t e^{-sQ }\Sigma e^{-sQ^T } \ud s \qquad \textrm{for} \quad t>0,
\] 
implying the integrability of  the map $\lambda \mapsto|\lambda|^k|\widehat{\mu^\natural}_t(\lambda)|$, for any  $k$ nonnegative integer,  and implicitly the smoothness for the densities of $X_t^{(1)}$ and $I_t^{\natural}$. Moreover, if the log-moment condition \eqref{cond1} holds, then the same holds true for the limiting distributions $\mu^{(1)}$ and $I^{\natural}_\infty$, where the Gaussian component is described by the covariance matrix 
\[
 \Sigma_\infty= \int_0^\infty e^{-sQ }\Sigma e^{-sQ^T } \ud s.
\] 
In order to deduce the second part of  condition {\bf (H)}, we first observe that the characteristic exponent $\psi$ of the L\'evy process $\xi$ satisfies
\[
\psi(\lambda)=\mathcal{O}(|\lambda|^2) \qquad \textrm{as} \quad |\lambda|\to \infty.
\]
Since   $\Re{(\psi(\lambda))}\leq 0$ for any $\lambda \in \mathbb{R}^d$,  there exist  constants $C>0$ and $R>0$ such that $\Re(\psi(\lambda))\leq -C |\lambda|^{2}$
 for any $|\lambda|\geq L>0$. 
 
On the other hand, recall that for any $Q\in \mathcal{M}^{+}(d)$ there are positive constants $c_1$, $c_2$, $c_3$ and $c_4$ such that
\begin{equation}\label{desQ}
c_4e^{-c_2 t}|\lambda|\leq |e^{-tQ^T}\lambda|\leq c_3e^{-c_1 t}|\lambda| \qquad \textrm{for }\quad t\geq 0 \quad\textrm{ and } \quad\lambda\in \mathbb{R}^d,
\end{equation}
 see for instance page $139$ of Urbanik \cite{UR}.  Hence, we take $R>c^{-1}_4L$ and introduce  $t_0(R):=c_2^{-1}\ln\left(\frac{Rc_4}{L}\right)>0$.
Putting all the pieces together,  we have for $t\ge t_0(R)$, that
\begin{align*}
\int_{D^c_R}|\widehat{\mu^\natural}_t(\lambda)| \ud \lambda  & = 
\int_{D^c_R}\exp\left\{\int_0^t 
\Re{\left( \psi( e^{-sQ^T } \lambda) \right) \ud s   }  \right\}\ud \lambda   \\
&\leq \int_{D^c_R}\exp\left\{\int_0^{t_0{(R)}} 
\Re{\left( \psi( e^{-sQ^T } \lambda) \right)\ud s   }  \right\}\ud \lambda  &\\
&\leq \int_{D^c_R}\exp\left\{-C\int_0^{t_0{(R)}} 
|e^{-sQ^T } \lambda| ^{2}\ud s    \right\}\ud \lambda &\\
&\leq \int_{D^c_R}\exp\left\{-C
\frac{c^{2}_4}{2 c_2}
\left(1-e^{-2 c_2 t_0{(R)}}\right)  |\lambda|^{2}    \right\}\ud \lambda,
\end{align*}
which implies the second part of  condition {\bf (H)}.


{\bf 2.}  The second case, that we consider here, is very similar to the previous example and includes the so-called  family of stable L\'evy processes. Indeed,  we suppose that there exists $\alpha\in (0,2)$ such that
\begin{equation}\label{cond9}
\limsup\limits_{|\lambda|\rightarrow \infty}\frac{\Re{(\psi(\lambda))}}{|\lambda|^\alpha}\in (-\infty,0).
\end{equation}
In other words,  there exist  constants $C>0$ and $R>0$ such that $\Re{(\psi(\lambda))}\leq -C|\lambda|^{\alpha}$
 for any $|\lambda|\geq L>0$.  
Hence, for any $t>0$, we define  $L_t=c^{-1}_4 Le^{c_2t}$ and  deduce
\begin{align*}
\int_{D^c_{L_t}}|\widehat{\mu^\natural}_t(\lambda)| \ud \lambda  & = 
\int_{D^c_{L_t}}\exp\left\{\int_0^t 
\Re{\left( \psi( e^{-sQ^T } \lambda) \right) \ud s   }  \right\}\ud \lambda   \\
&\leq \int_{D^c_{L_t}}\exp\left\{-C\int_0^{t} 
|e^{-sQ^T } \lambda| ^{\alpha}\ud s    \right\}\ud \lambda &\\
&\leq \int_{D^c_{L_t}}\exp\left\{-C
\frac{c^{2}_4}{2 c_2}
\left(1-e^{-\alpha c_2 t}\right)  |\lambda|^{\alpha}    \right\}\ud \lambda.
\end{align*}
The previous integral  is clearly finite and implicitly the smoothness for the densities of 
$X^{(1)}_t$ and $I^{\natural}_t$ is obtained. If the log-moment condition \eqref{cond1} holds, then the same holds true for the limiting distributions $\mu^{(1)}$ and $I^{\natural}_\infty$ and 
implicitly, the smoothness for their densities.
 
Using exactly the same arguments but with $R>c^{-1}_4 L$ and   $t_0(R):=c_2^{-1}\ln\left(\frac{Rc_4}{L}\right)>0$, we deduce for $t\geq t_0(R)$ 
\[\int_{D^c_R}|\widehat{\mu^\natural}_t(\lambda)| \ud \lambda \leq\int_{D^c_R}\exp\left\{-C
\frac{c^{\alpha}_4}{\alpha c_2}
\left(1-e^{-\alpha c_2 t_0{(R)}}\right)  |\lambda|^{\alpha}    \right\}\ud \lambda,
\]
which implies the second part of  condition {\bf (H)}.  


{\bf  3.} Our third case  impose an {\it Orey-Masuda} or {\it Kallenberg-Bornarchuk-Kulik} type condition on the jump structure of the L\'evy process $\xi$. To be more precise, let us  assume that  there exists a radial non-negative function $\kappa:\mathbb{R}^d \to [0,\infty)$  satisfying
\begin{itemize}
\item[i)] as a function of the radius, i.e. $\tilde{\kappa}(r)=\kappa(v)$ if $|v|=r>0$, it is non-decreasing, 
\item[ii)] for any  $\beta>0$, we have
\[
\int_{D_1^c}{e^{-\beta \kappa(v)}}\ud v<\infty,
\]
\item[iii)]  and 
\begin{equation}\label{OMtype}
\int\limits_{\{z\in \mathbb{R}^d: |\langle z,v \rangle|\leq 1\}}\langle z,v \rangle ^2 \nu(\ud z)\geq \kappa(v)\qquad 
\textrm{ for any } v\in \mathbb{R}^d  \quad\textrm{ with }\quad  |v|\geq 1.
\end{equation}
\end{itemize}

The previous assumption  on the L\'evy measure $\nu$ is an  {\it Orey-Masuda}  or {\it  Kallenberg-Bornarchuk-Kulik} type condition. Indeed, we observe  that {\it Orey-Masuda's} condition \eqref{Masuda'scond} is fulfilled when $\kappa(v)=c|v|^{\alpha}$ with $\alpha\in (0,2)$ and $c>0$. Similarly, {\it Kallenberg's} condition \eqref{kalle} and {\it Bornarchuk-Kulik's} condition \eqref{borna} are satisfied when $\kappa(v)=g(|v|)\ln (|v|)$ and $g$ is increasing and  goes to $\infty$, as $|v|$ goes to $\infty$. It is important to note that our assumptions does not seem to imply  condition  \eqref{BodKcond}.

In order to prove that for any $t>0$, $\mu^\natural_t$ possesses a density which is smooth, we first observe that 
  for any $\lambda\in \mathbb{R}^d$
\[
|\widehat{\mu^\natural}_t(\lambda)|\leq \left|\exp\left\{\int_{0}^{t}\int_{\mathbb{R}^d} 
\left(e^{i\langle e^{-sQ^T}\lambda,z\rangle}-1-i\langle e^{-sQ^T}\lambda,z\rangle 1_{\{|z|\le 1\}}\right)\nu(\ud z) \ud s   \right\}\right|.
\]
 Recalling that $1-\cos(x)\geq (\nicefrac{2}{\pi^{-2}})x^2$ for $|x|\leq \pi$, we deduce that for any $\lambda\in \mathbb{R}^d$,
 \[
\begin{split}
|\widehat{\mu^\natural}_t(\lambda)|&\le
\exp\left\{\int_{0}^{t}\int_{\{ |\langle z,e^{-sQ^T}\lambda \rangle |\leq \pi\}} 
\left(\cos\left(\langle e^{-sQ^T}\lambda,z\rangle\right)-1\right)\nu(\ud z) \ud s   \right\}  \\
&\le \exp\left\{-2\int_{0}^{t}\int_{\{ |\langle z,e^{-sQ^T}\lambda \rangle |\leq \pi\}} 
{\langle z,\nicefrac{e^{-sQ^T}\lambda}{\pi} \rangle}{}^2\nu(\ud z) \ud s \right\}.
\end{split}
\]
 Next, we take $L_t=L\pi c_4^{-1}e^{c_2 t}$, and 
 observe  that if $|\lambda|\geq L$ and $0\leq s\leq t$, then 
${|e^{-sQ^T}\lambda |}\geq \pi$. Hence, using \eqref{OMtype} together with \eqref{desQ} and the fact that $\tilde\kappa$ is non-decreasing, we deduce 
\[
\begin{split}
\int_{D^c_{L_t}} |\widehat{\mu^\natural}_t(\lambda)|\ud\lambda & \leq \int_{D^c_{L_t}}\exp\left\{-2\int_{0}^{t}\int_{\{ |\langle z,e^{-sQ^T}\lambda \rangle |\leq \pi\}} 
{\langle z,\nicefrac{e^{-sQ^T}\lambda}{\pi} \rangle}{}^2\nu(\ud z) \ud s \right\}\ud\lambda \\
&\le \int_{D^c_{L_t}}  \exp\left\{-2\int_{0}^{t}\kappa\big(|\nicefrac{e^{-sQ^T}\lambda}{\pi}|\big)
 \ud s \right\}\ud\lambda \\
&\le\int_{D^c_{L_t}}  \exp\left\{-2{t}\kappa\big(c_4 e^{c_2 t} \pi^{-1}|\lambda|\big)
 \right\}\ud\lambda\\
&\le\int_{D^c_{L}}  \exp\left\{-2{t}\kappa\big(|\lambda|\big)
 \right\}\ud\lambda,\\
\end{split}
\]
which is integrable  from our hypothesis. Similarly as in the previous cases, the latter implies  that the densities of $X_t^{(1)}$ and $I_t^{\natural}$, for $t\geq 0$,  belong to $\mathcal{C}_b$.   Moreover, if the log-moment condition \eqref{cond1} holds, then the same holds true for the limiting distributions $\mu^{(1)}$ and $I^{\natural}_\infty$.

Using exactly the same arguments but with $R>c_4^{-1}\pi$ and   $t_0(R):=c_2^{-1}\ln\left(\frac{c_4R}{\pi}\right)>0$, we deduce that for any $t\geq t_0(R)$
\[\int_{D^c_R}|\widehat{\mu^\natural}_t(\lambda)| \ud \lambda =\int_{D^c_R}\exp\left\{-2{t_0}(R)\kappa\big(c_4 \pi^{-1}e^{-c_2t_0(R)}|\lambda|\big) \right\}\ud \lambda,
\]
which after change of variable and using our assumptions  on $\kappa$, allow us to deduce  the second part of  condition {\bf (H)}.


\section{Proofs}\label{proofs}
\subsection{Preliminaries}
From the so-called L\'evy-It\^o decomposition, we can express the L\'evy processes $(\xi_t, t\geq 0)$  as the sum of two independent L\'evy processes, in other words,
\begin{equation*}\label{levyito}
\xi_t:=at+\sqrt{\Sigma} B_t+\xi^{(1)}_t \qquad \textrm{for}\quad t\geq 0,
\end{equation*}
where we recall that $a\in\mathbb{R}^d$, $\Sigma$ is a $d$-squared symmetric non-negative definite matrix, $\sqrt{\Sigma}$ is any matrix such that $\langle\sqrt{\Sigma} z, \sqrt{\Sigma} z\rangle=\langle z, \Sigma z\rangle$ for $z\in \mathbb{R}^d$, $B=(B_t, t\geq 0)$ is a $d$-dimensional Brownian motion and $\xi^{(1)}=(\xi^{(1)}_t, t\geq 0)$ is  a pure jump L\'evy process in $\mathbb{R}^d$ which is independent of $B$.
The latter implies that we  can rewrite the solution of the SDE $(\ref{sdem})$ as follows
\begin{equation*}\label{sdetv1}
X^{(\epsilon)}_t=e^{-tQ }x_0+\sqrt{\epsilon}a\int\limits_{0}^{t}{e^{-(t-s)Q}\ud s}+
\sqrt{\epsilon}\int\limits_{0}^{t}{e^{-(t-s)Q}\ud (\sqrt{\Sigma} B_s+ \xi^{(1)}_s)}\qquad \textrm{for} \quad t\geq 0.
\end{equation*}
Since $\xi^\natural_t=\sqrt{\Sigma} B_t+\xi^{(1)}_t$, for $t\ge 0$,  we identify
\begin{equation*}
I^\natural_t=\int\limits_{0}^{t}{e^{-(t-s)Q}\ud(\sqrt{\Sigma}B_s+ \xi^{(1)}_s)} \qquad \textrm{ for }\quad t\ge 0,
\end{equation*}
and for simplicity, we write $C_t:=\left({I-e^{- tQ}}\right)Q^{-1}a$, for $t\ge 0$.
Then, we write $X^{(\epsilon)}$ as follows
\begin{equation*}\label{sdetv11}
X^{(\epsilon)}_t=e^{- tQ}x_0+\sqrt{\epsilon}C_t+\sqrt{\epsilon}I^\natural_t \qquad \textrm{for} \quad t\geq 0.
\end{equation*}
Assuming that $\xi$ satisfies  the log-moment condition (\ref{cond1}), then $X^{(\epsilon)}_t$ converges in distribution to $X^{(\epsilon)}_\infty$, as $t$ goes to $\infty$. We recall that the  law of  $X^{(\epsilon)}_\infty$ is given by $\mu^{(\epsilon)}$ and  observe that it can be written as
\[
X^{(\epsilon)}_\infty=\sqrt{\epsilon}C_\infty+\sqrt{\epsilon}I^\natural_\infty,
\]
where  $I^\natural_\infty$ denotes the  limiting distribution  of  $I^\natural_t$  as $t$ increases. We also recall that  $I^\natural_\infty$ is 
 $Q$-self-decomposable and, if it is non-degenerate, then its distribution is absolutely continuous with respect to the Lebesgue measure on $\mathbb{R}^d$ (see Yamazato \cite{Ya}). 

Bearing  all this in mind, in order to study  how the process $X^{(\epsilon)}$ converges to its equilibrium distribution $\mu^{(\epsilon)}$ under the total variation distance, as $t$ increases,   we define the auxiliary metric as follows
\begin{equation}\label{aux}
D^{(\epsilon)}(t):=\norm{\left(\frac{1}{\sqrt{\epsilon}}e^{-tQ}x_0+I^\natural_\infty\right)-I^\natural_\infty},
\end{equation}
and introduce the error term 
\begin{equation}\label{errores}
R(t):=\norm{\left(C_t+I^\natural_t\right)-X^{(1)}_\infty},
\end{equation}
which does not depend on $\epsilon$.
\begin{lemma}\label{lemmadif}For any  $\epsilon>0$ and $t> 0$, we have
\begin{equation}\label{cutu1}
\left|d^{(\epsilon)}(t)-D^{(\epsilon)}(t)\right|\leq R(t).
\end{equation}
\end{lemma}
\begin{proof} We first use the triangle inequality to deduce
\begin{equation}\label{today}
\begin{split} 
d^{(\epsilon)}(t)&\leq \norm{\left(e^{- t Q}x_0+X^{(\epsilon)}_\infty\right)-X^{(\epsilon)}_\infty}\\
&+  \norm{\left(e^{- t Q}x_0+\sqrt{\epsilon}C_t+\sqrt{\epsilon}I^\natural_t\right)-\left(e^{- t Q}x_0+X^{(\epsilon)}_\infty\right)}.
\end{split}
\end{equation}
On the one hand,  from  Lemma \ref{sca} part (i), we see
\[
\norm{\left(e^{- t Q}x_0+\sqrt{\epsilon}C_t+\sqrt{\epsilon}I^\natural_t\right)-\left(e^{- t Q}x_0+X^{(\epsilon)}_\infty\right)}=\norm{\left(\sqrt{\epsilon}C_t+\sqrt{\epsilon}I^\natural_t\right)-X^{(\epsilon)}_\infty}.
\]
Recalling that  $X^{(\epsilon)}_\infty=\sqrt{\epsilon}C_\infty+\sqrt{\epsilon}I^\natural_\infty$, then from  Lemma \ref{sca}  part (ii),  we have
\[
\norm{\left(\sqrt{\epsilon}C_t+\sqrt{\epsilon}I^\natural_t\right)-X^{(\epsilon)}_\infty}=\norm{\left(C_t+I^\natural_t\right)-
\left(C_\infty+I^\natural_\infty \right)
}=R(t).
\]
On the other hand since $X^{(\epsilon)}_\infty=\sqrt{\epsilon}X^{(1)}_\infty$, we apply   Lemma \ref{sca} part (iii) to deduce
\[
\norm{\left(e^{- t Q}x_0+X^{(\epsilon)}_\infty\right)-X^{(\epsilon)}_\infty}=\norm{\left(\frac{1}{\sqrt{\epsilon}}e^{-t Q}x_0+X^{(1)}_\infty\right)-X^{(1)}_\infty}=D^{(\epsilon)}(t).
\]
Putting all pieces together in inequality \eqref{today}, allow us to get the following inequality
\[
d^{(\epsilon)}(t)\leq R(t)+D^{(\epsilon)}(t) \qquad \textrm{for} \quad t> 0.
\]
Similarly,  we have
\[
\begin{split}
D^{(\epsilon)}(t)&=\norm{\left(e^{- tQ}x_0+\sqrt{\epsilon}I^\natural_\infty\right)-\sqrt{\epsilon}I^\natural_\infty}\\
&\leq \norm{\left(e^{- tQ}x_0+\sqrt{\epsilon}C_\infty+\sqrt{\epsilon}I^\natural_\infty\right)-X_t^{(\epsilon)}}+\norm{X_t^{(\epsilon)}-X_\infty^{(\epsilon)}}\\
& =R(t)+d^{(\epsilon)}(t),\\
\end{split}
\]
where the first  and the last identities follow from   Lemma \ref{sca} (parts (i), (ii) and (iii)); and  the second inequality follows from the triangle inequality.
The desired result now follows form both inequalities.
\end{proof}

Therefore, our approach for proving the main result consists in determining the cut-off phenomenon (window  and/or profile respectively) for the auxiliary distance \eqref{aux} and  that the 
error term \eqref{errores} vanishes, as $t$ goes to infinity. The latter, together with inequality \eqref{cutu1}  implies the cut-off phenomenon  for the distance $d^{(\epsilon)}$ as we will see in Subsection \ref{pruebita}.

\subsection{Auxiliary metric}
For a better understanding of our arguments,  we study separately the behaviour of the auxiliary metric $D^{(\epsilon)}$. In order to do so, we use  the asymptotic behaviour in Lemma \ref{jara}.

\begin{proposition}\label{gen2}
Let $Q\in \mathcal{M}^+(d)$. Using the same notation as  in  Lemma \ref{jara}, we  let $0<\epsilon\ll 1$ 
such that 
\[
t_\epsilon:=(2\gamma)^{-1}\ln\left(\nicefrac{1}{\epsilon}\right)+
\frac{\ell-1}{\gamma}
\ln\left(\ln\left(\nicefrac{1}{\epsilon}\right)\right)>0\quad
\mathrm{  and  }\quad
w_\epsilon:=\gamma^{-1}+o_\epsilon(1)>0,
\]
where $\lim\limits_{\epsilon\rightarrow 0}o_\epsilon(1)=0$. Moreover, if the density of the random variable $I^\natural_\infty$  is  continuous almost everywhere with respect to the Lebesgue measure and
\begin{itemize}
\item[i)]  if $\theta_k=0$ for every $k\in \{1,\ldots,m\}$, then
\[
\begin{split}
\lim\limits_{\epsilon\rightarrow 0}D^{(\epsilon)}(t_\epsilon+cw_\epsilon)=
\norm{\left((2\gamma)^{1-\ell}e^{-c}v(x_0)+I^\natural_\infty\right)-I^\natural_\infty},
\end{split}
\]
for every $c\in \mathbb{R}$, where $v(x_0):=\sum_{k=1}^{m} v_k$,
\item[ii)] 
if $\theta_k\not=0$ for some $k\in \{1,\ldots,m\}$, we have for each $c\in \mathbb{R}$ that there exists subsequences
$(t_{\epsilon^{\prime}}+cw_{\epsilon^{\prime}}, \epsilon^{\prime}>0)$ of $(t_\epsilon+cw_\epsilon, \epsilon>0)$ such that
\[
\tilde{v}_c(x_0):=\lim\limits_{{\epsilon^{\prime}}\rightarrow 0}\sum\limits_{k=1}^{m} e^{i  (t_{\epsilon^{\prime}}+cw_{\epsilon^{\prime}})\theta_k} v_k \qquad \textrm{exists,}
\]
and is different from the null vector; and moreover 
\[
\begin{split}
\lim\limits_{\epsilon^{\prime}\rightarrow 0} D^{(\epsilon^{\prime})}(t_{\epsilon^{\prime}}+cw_{\epsilon^{\prime}})=
\norm{\left((2\gamma)^{1-\ell}e^{-c}\tilde{v}_c(x_0)+I^\natural_\infty\right)-I^\natural_\infty}.
\end{split}
\]
\end{itemize}
\end{proposition}
\begin{proof} We first prove part  (i). Let us  define
 \[\widetilde{R}^{(\epsilon)}(t):=\norm{\left(\frac{1}{\sqrt{\epsilon}}{e^{-tQ }x_0}+I^\natural_\infty\right)-\left(\frac{t^{\ell-1}}{e^{\gamma t}\sqrt{\epsilon}}\sum\limits_{k=1}^{m} e^{i \theta_k t} v_k+I^\natural_\infty\right)}, \] and
 \[\widetilde{D}^{(\epsilon)}(t):=\norm{\left(\frac{t^{\ell-1}}{e^{\gamma t}\sqrt{\epsilon}}\sum\limits_{k=1}^{m} e^{i \theta_k t} v_k+I^\natural_\infty\right)-I^\natural_\infty}.\]
 Recalling the definition of the auxiliary metric in \eqref{aux} and using the triangle inequality, we deduce
\[
\begin{split}
D^{(\epsilon)}(t)
&\leq  \norm{\left(\frac{t^{\ell-1}}{e^{\gamma t}}\sum\limits_{k=1}^{m} e^{i \theta_k t} v_k+\sqrt{\epsilon}I^\natural_\infty\right)-\sqrt{\epsilon}I^\natural_\infty}\\
& +\norm{\left(e^{- tQ}x_0+\sqrt{\epsilon}I^\natural_\infty\right)-\left(\frac{t^{\ell-1}}{e^{\gamma t}}\sum\limits_{k=1}^{m} e^{i \theta_k t} v_k+\sqrt{\epsilon}I^\natural_\infty\right)}.
\end{split}
\]
We then apply Lemma \ref{sca} part (iii) in order to deduce 
$D^{(\epsilon)}(t)\leq 
\widetilde{R}^{(\epsilon)}(t)+\widetilde{D}^{(\epsilon)}(t)$.
On the other hand, using again the triangle inequality, we obtain
\begin{align*}
\widetilde{D}^{(\epsilon)}(t)
&\leq  \norm{\left(\frac{e^{- tQ}x_0}{\sqrt{\epsilon}}+I^\natural_\infty\right)-I^\natural_\infty}\\
& +\norm{\left(\frac{t^{\ell-1}}{e^{\gamma t}\sqrt{\epsilon}}\sum\limits_{k=1}^{m} e^{i \theta_k t} v_k+I^\natural_\infty\right)-\left(\frac{e^{- tQ}x_0}{\sqrt{\epsilon}}+I^\natural_\infty\right)}.
\end{align*}
Similarly as before, we apply Lemma \ref{sca} part  (iii) and deduce
$
\widetilde{D}^{(\epsilon)}(t)\leq 
\widetilde{R}^{(\epsilon)}(t)+D^{(\epsilon)}(t)$.
Putting all pieces together,  we get
\begin{equation}\label{desi1}
\left| D^{(\epsilon)}(t)-\widetilde{D}^{(\epsilon)}(t)  \right|\leq \widetilde{R}^{(\epsilon)}(t).
\end{equation}
Next, from  Lemma \ref{sca} part (i), we observe
 \[\widetilde{R}^{(\epsilon)}(t)
=\norm{\left(\frac{t^{\ell-1}}{e^{\gamma t}\sqrt{\epsilon}}\left(\frac{e^{\gamma t}e^{-tQ }x_0}{t^{\ell-1}}-\sum\limits_{k=1}^{m} e^{i \theta_k t} v_k\right)+I^\natural_\infty\right)-I^\natural_\infty}.\]
On the other hand, straightforward computations led us to
\begin{equation}\label{limite}
\lim\limits_{\epsilon\rightarrow 0}\frac{(t_\epsilon+cw_\epsilon)^{\ell-1}{e^{-\gamma (t_\epsilon+cw_\epsilon)}}}{\sqrt{\epsilon}}=(2\gamma)^{1-\ell}e^{-c},
\end{equation}
for any $c\in \mathbb{R}$, which implies, together with
Lemma \ref{jara}, that 
\[
\lim\limits_{\epsilon\rightarrow 0}\frac{(t_\epsilon+cw_\epsilon)^{\ell-1}}{e^{\gamma (t_\epsilon+cw_\epsilon)}\sqrt{\epsilon}}\left(\frac{e^{\gamma (t_\epsilon+cw_\epsilon)}e^{-(t_\epsilon+cw_\epsilon)Q }x_0}{(t_\epsilon+cw_\epsilon)^{\ell-1}}-\sum\limits_{k=1}^{m} e^{i \theta_k (t_\epsilon+cw_\epsilon)} v_k\right)=0
\]
for every $c\in \mathbb{R}$.
Therefore, Scheff\'e's  Lemma allow us to deduce
\begin{equation}
\label{resid}
\lim\limits_{\epsilon\rightarrow 0}{\widetilde{R}^{(\epsilon)}(t_\epsilon+cw_\epsilon)}=0
\;\; \textrm{ for any } c\in \mathbb{R}.
\end{equation}
Since $\theta_k=0$ for every $k\in \{1,\ldots,m\}$,  Scheff\'e's  Lemma together with relation \eqref{limite} imply
\[
\lim\limits_{\epsilon\rightarrow 0}\widetilde{D}^{(\epsilon)}(t_\epsilon+cw_\epsilon)=\norm{\left((2\gamma)^{1-\ell}e^{-c}v(x_0)+I^\natural_\infty\right)-I^\natural_\infty},
\]
for every $c\in \mathbb{R}$.
Using inequality \eqref{desi1} and Scheff\'e's  Lemma again, we deduce
\[
\lim\limits_{\epsilon\rightarrow 0}{D}^{(\epsilon)}(t_\epsilon+cw_\epsilon)=\norm{\left((2\gamma)^{1-\ell}e^{-c}v(x_0)+I^\natural_\infty\right)-I^\natural_\infty}=
G_{x_0}(c)
\]
for every $c\in \mathbb{R}$.
Finally, we use again Scheff\'e's Lemma to derive  $\lim\limits_{c \rightarrow \infty}{G_{x_0}(c)}=0$.
Using Lemma \ref{a5}, we obtain $\lim\limits_{c \rightarrow -\infty}{G_{x_0}(c)}=1$. The proof of part (i) is now complete.

The proof of part (ii) follows from similar arguments as those used above by taking a subsequence $(t_{\epsilon^{\prime}}+cw_{\epsilon^\prime}, \epsilon^{\prime}>0)$ of the sequence $(t_\epsilon+cw_\epsilon, \epsilon>0)$. Indeed, we first observe that  inequality \eqref{desi1} and the limit \eqref{resid}  always holds.  On the one hand, we also observe that
\begin{equation}\label{cebollon}
\lim_{\epsilon\rightarrow 0}\sum\limits_{k=1}^{m} e^{i \theta_k (t_{\epsilon}+cw_{\epsilon})} v_k \qquad \textrm{may not exist.}
\end{equation}

On the other hand, from Lemma \ref{jara} we have
\[
0<\liminf_{t\to\infty}\left|\sum\limits_{k=1}^{m} e^{i \theta_k t} v_k\right|\le \limsup_{t\to\infty}\left|\sum\limits_{k=1}^{m} e^{i \theta_k t} v_k\right|<\infty,
\]
which allows us to deduce that the null vector is not in the basin of attraction which is defined by
\begin{equation}\label{Bassin}
\mathrm{Bas}:=\left\{v\in \mathbb{R}^d: \textrm{there exists a sequence } (t_j)\uparrow \infty  \textrm{ and } \lim_{t_j\to\infty}\sum\limits_{k=1}^{m} e^{i \theta_k t_j} v_k =v\right\},
\end{equation}
and turns out to be  not empty since the vectors $\{v_k, k=1,\ldots, m\}$ are linearly independent.

Next, let 
$\overline{D}=\limsup\limits_{\epsilon \to 0}D^{(\epsilon)}(t_\epsilon+cw_\epsilon)$ and we take  a subsequence $(t_{\epsilon^{\prime}}+cw_{\epsilon^\prime}, \epsilon^{\prime}>0)$ of $(t_\epsilon+cw_\epsilon, \epsilon>0)$
 such that 
$\overline{D}=\lim\limits_{\epsilon^{\prime} \to 0}D^{(\epsilon^{\prime})}(t_{\epsilon^{\prime}}+cw_{\epsilon^{\prime}})$.
We write 
\[
v(t,x_0)=\sum\limits_{k=1}^{m} e^{i \theta_k t} v_k\
\]
and deduce $|v(t_{\epsilon^{\prime}}+cw_{\epsilon^{\prime}},x_0)|\leq \sum\limits_{k=1}^{m} |v_k|$ for any $\epsilon^\prime> 0$. By  Bolzano-Weierstrass' Theorem, we get
that there 
exists a subsequence $(t_{\epsilon^{\prime\prime}}+cw_{\epsilon^{\prime\prime}}, \epsilon^{\prime\prime}>0)$ of  $(t_{\epsilon^{\prime}}+cw_{\epsilon^{\prime}}, \epsilon^{\prime}>0)$
such that 
$\lim\limits_{\epsilon^{\prime\prime} \to 0} v(t_{\epsilon^{\prime\prime}}+cw_{\epsilon^{\prime\prime}},x_0)=\tilde{v}(x_0)\in \mathrm{Bas}$.
From Scheff\'e's Lemma, we get
\[
\begin{split}
\lim\limits_{\epsilon^{\prime\prime}\rightarrow 0} D^{(\epsilon^{\prime\prime})}(t_{\epsilon^{\prime\prime}}+cw_{\epsilon^{\prime\prime}})=
\norm{\left((2\gamma)^{1-\ell}e^{-c}\tilde{v}(x_0)+I^\natural_\infty\right)-I^\natural_\infty}.
\end{split}
\]
Since 
$\overline{D}=\lim\limits_{\epsilon^{\prime\prime} \to 0}D^{(\epsilon^{\prime\prime})}(t_{\epsilon^{\prime\prime}}+cw_{\epsilon^{\prime\prime}})$, 
then 
\[
\limsup\limits_{\epsilon \to 0}D^{(\epsilon)}(t_\epsilon+cw_\epsilon)=\norm{\left((2\gamma)^{1-\ell}e^{-c}\tilde{v}(x_0)+I^\natural_\infty\right)-I^\natural_\infty}.
\]
The proof of part (ii) is now complete.
\end{proof}
It is important to note that \eqref{cebollon} breaks down the existence of a profile function and only windows cut-off for the auxiliary distance $D^{(\epsilon)}$ can be hoped. Indeed, from the  previous proof, we have deduced
\begin{equation}\label{limsupcebollon}
\limsup\limits_{\epsilon \to 0}D^{(\epsilon)}(t_\epsilon+cw_\epsilon)=\norm{\left((2\gamma)^{1-\ell}e^{-c}\tilde{v}(x_0)+I^\natural_\infty\right)-I^\natural_\infty},
\end{equation}
where $\tilde{v}(x_0)\in \mathrm{Bas}$. Similarly,  we can obtain 
\begin{equation}\label{liminfcebollon}
\lim\inf\limits_{\epsilon \to 0}D^{(\epsilon)}(t_\epsilon+cw_\epsilon)=\norm{\left((2\gamma)^{1-\ell}e^{-c}\hat{v}(x_0)+I^\natural_\infty\right)-I^\natural_\infty},
\end{equation}
where $\hat{v}(x_0)\in \mathrm{Bas}$.
Since $\lim\limits_{c\to \infty}e^{-c}\tilde{v}(x_0)=0$, we use Scheff\'e's Lemma and   conclude
\[
\lim\limits_{c\to \infty}\limsup\limits_{\epsilon \to 0}D^{(\epsilon)}(t_\epsilon+cw_\epsilon)=0.
\]
By observing  that $\lim\limits_{c\to -\infty}\left|e^{-c}\hat{v}(x_0)\right|=\infty$,
then Lemma \ref{a5} implies
\[
\lim\limits_{c\to -\infty}\liminf\limits_{\epsilon \to 0}D^{(\epsilon)}(t_\epsilon+cw_\epsilon)=1.
\]

\subsection{Proof of the main result}\label{pruebita}
\begin{proposition}\label{jcppp11}
Assume  that  the log-moment condition \eqref{cond1} and  {\bf (H)} hold,  then
\[
\lim\limits_{t\to \infty}R(t)=\lim_{t\to \infty}\norm{\left(C_t+I^\natural_t\right)-X_\infty^{(1)}}=0.
\]
\end{proposition}

\begin{proof} We first observe that  $X_\infty^{(1)}$ has the same distribution as $C_\infty+I^\natural_\infty$. From the triangle inequality, we have
\[
\begin{split}
\norm{\left(C_t+I^\natural_t\right)-X_\infty^{(1)}} \leq &
\norm{\left(C_t+I^\natural_t\right)-\left(C_t+I^\natural_\infty\right)}+\norm{\left(C_t+I^\natural_\infty\right)-\left(C_\infty+I^\natural_\infty\right)}.
\end{split}
\]
From our assumptions $I^\natural_\infty$ has  a continuous  density  and since $\lim\limits_{t\rightarrow \infty}C_t=C_\infty$, 
an application of   Scheff\'e's Lemma  allows us to deduce
\[
\lim\limits_{t\rightarrow \infty}\norm{\left(C_t+I^\natural_\infty\right)-\left(C_\infty+I^\natural_\infty\right)}=0.
\]
It remains to prove
\[
\lim\limits_{t\rightarrow \infty}\norm{\left(C_t+I^\natural_t\right)-\left(C_t+I^\natural_\infty\right)}=0,
\]
which is equivalent, according to Lemma \ref{sca} part (i), to
\[
\lim\limits_{t\rightarrow \infty}\norm{I_t^\natural-I^\natural_\infty}=0.
\]
From our hypothesis,  we have that $I^\natural_t$ has a continuous density $f_t(x)$ that goes to $0$ as $|x|$ goes to $\infty$. Recalling that 
\[
|\widehat{\mu^\natural}_\infty(\lambda)|\leq |\widehat{\mu^\natural}_{t}(\lambda)| \qquad \textrm{ for any }\quad t>0\quad \textrm{ and } \quad \lambda\in \mathbb{R}^d,
\]
we also deduce that $I^\natural_\infty$ has a continuous density $f_\infty(x)$ that goes to $0$ as $|x|$ goes to $\infty$, under our assumptions.
By the Fourier inversion formula, we know, for Lebesgue almost everywhere $x\in \mathbb{R}^d$, 
\begin{equation*}
f_t(x)=\frac{1}{\left(2\pi\right)^{d}}\int_{\mathbb{R}^d}{e^{-i\langle x,\lambda \rangle}\widehat{\mu^\natural}_{t}(\lambda)\ud \lambda}\quad \textrm{and}\quad
f_\infty(x)=\frac{1}{\left(2\pi\right)^{d}}\int_{\mathbb{R}^d}{e^{-i\langle x,\lambda \rangle}\widehat{\mu^\natural}_\infty(\lambda)\ud\lambda}.
\end{equation*}
Therefore,   for Lebesgue almost everywhere $x\in \mathbb{R}^d$, 
\begin{equation*}
\left  |f_t(x)-f_\infty(x)\right | \leq 
\frac{1}{\left(2\pi\right)^{d}}\int_{\mathbb{R}^d}{\left|\widehat{\mu^\natural}_{t}(\lambda)-\widehat{\mu^\natural}_\infty(\lambda)\right| \ud \lambda}.
\end{equation*}
If 
\begin{equation}\label{goal0}
\lim\limits_{t\rightarrow \infty}\int_{\mathbb{R}^d}{\left |\widehat{\mu^\natural}_{t}(\lambda)-\widehat{\mu^\natural}_\infty(\lambda)\right| \ud\lambda}=0,
\end{equation} 
then $\lim\limits_{t\rightarrow \infty}f_t(x)=f_\infty(x)$, for Lebesgue almost everywhere $x\in \mathbb{R}^d$. Hence Scheff\'e's Lemma allow us to deduce
\begin{equation*}
\lim\limits_{t\rightarrow \infty}\norm{I_t^\natural-I^\natural_\infty}=0.
\end{equation*}
In other words, the proof will be completed if we deduce \eqref{goal0}. In order to do so, we take $R>0$ and  introduce a strictly positive constant $t_0(R)$ that only depends on $R$. Thus, we observe
\begin{equation}\label{eqteo}
\int_{\mathbb{R}^d}{\left |\widehat{\mu^\natural}_{t}(\lambda)-\widehat{\mu^\natural}_\infty(\lambda)\right| \ud\lambda}=\int_{D_{R}}{\left |\widehat{\mu^\natural}_{t}(\lambda)-\widehat{\mu^\natural}_\infty(\lambda)\right| \ud\lambda}+\int_{D^c_{R}}{\left |\widehat{\mu^\natural}_{t}(\lambda)-\widehat{\mu^\natural}_\infty(\lambda)\right| \ud\lambda},
\end{equation} 
for $t\ge t_0(R)$,
where we recall $D_R=\{z\in \mathbb{R}^d: |z|\leq R\}$ and $D^{c}_R=\mathbb{R}^d\setminus {D_R}$.
Since $I^\natural_t$ converges in distribution to $I^\natural_\infty$  as $t$ goes to infinity  then $\widehat{\mu^\natural}_{t}(\cdot)$ converges uniformly on compact sets to 
$\widehat{\mu^\natural}_\infty(\cdot)$ as $t$ goes to infinity. Then, for any $R>0$ we have
\begin{eqnarray*}
\lim\limits_{t\rightarrow  \infty}\int_{D_{R}}{\left |\widehat{\mu^\natural}_{t}(\lambda)-\widehat{\mu^\natural}_\infty(\lambda)\right| \ud\lambda}=0.
\end{eqnarray*}
On the other hand for the second term in \eqref{eqteo}, we get, for any $R>0$, 
\begin{eqnarray*}
\limsup\limits_{t\rightarrow  \infty}\int_{D^c_{R}}{\left |\widehat{\mu^\natural}_{t}(\lambda)-\widehat{\mu^\natural}_\infty(\lambda)\right| \ud\lambda}
\leq \limsup\limits_{t\rightarrow  \infty}\int_{D^c_{R}}{\left |\widehat{\mu^\natural}_{t}(\lambda)\right| \ud\lambda}
+
\int_{D^c_{R}}{\left |\widehat{\mu^\natural}_\infty(\lambda)\right| \ud\lambda}.
\end{eqnarray*}
Since for any $t>t_0(R)$ we have
\begin{equation*}
\int_{D^c_{R}}{\left |\widehat{\mu^\natural}_\infty(\lambda)\right| \ud\lambda} \leq \int_{D^c_{R}}{\left |\widehat{\mu^\natural}_{t}(\lambda)\right| \ud\lambda}\le \sup_{s>t_0(R)}\int_{D^c_{R}}{\left |\widehat{\mu^\natural}_{s}(\lambda)\right| \ud\lambda}, 
\end{equation*}
we deduce
\begin{eqnarray*}
\limsup\limits_{t\rightarrow  \infty}\int_{D^c_{R}}{\left |\widehat{\mu^\natural}_{t}(\lambda)-\widehat{\mu^\natural}_\infty(\lambda)\right| \ud\lambda}\le 2\sup_{s>t_0(R)}\int_{D^c_{R}}{\left |\widehat{\mu^\natural}_{s}(\lambda)\right| \ud\lambda}.
\end{eqnarray*}
The latter inequality, together with our assumption ({\bf H}), imply  that as $R$ increases, 
\[
\limsup\limits_{t\rightarrow  \infty}\int_{D^c_{R}}{\left |\widehat{\mu^\natural}_{t}(\lambda)-\widehat{\mu^\natural}_\infty(\lambda)\right| \ud\lambda}= 0, 
\]
and implicitly we obtain   \eqref{goal0}. The proof is now complete.
\end{proof}
At this stage, we have all the tools to prove Theorem \ref{teo2}.
\begin{proof}[Proof of Theorem \ref{teo2}]
 We first prove part (i).  From Lemma \ref{lemmadif}, we have
\begin{align}\label{sand}
\left|d^{(\epsilon)}(t_\epsilon+cw_\epsilon)-D^{(\epsilon)}(t_\epsilon+cw_\epsilon)
\right|\leq R(t_\epsilon+cw_\epsilon),
\end{align}
 and from Proposition \ref{jcppp11}, we know that 
$\lim\limits_{\epsilon\rightarrow 0}R(t_\epsilon+cw_\epsilon)=0$, for any $c\in \mathbb{R}$.
On the other hand, from Proposition \ref{gen2} part (i), we also know
\[
\begin{split}
\lim\limits_{\epsilon\rightarrow 0}D^{(\epsilon)}(t_\epsilon+cw_\epsilon)=
\norm{\left((2\gamma)^{1-\ell}e^{-c}v(x_0)+I^\natural_\infty\right)-I^\natural_\infty},
\end{split}
\]
for any $c\in \mathbb{R}$. Putting all pieces together in inequality \eqref{sand},  the desired result is obtained.

Now, we prove part (ii). Recall from the equalities \eqref{limsupcebollon} and \eqref{liminfcebollon} that
\[
\begin{split}
\liminf\limits_{\epsilon^{}\rightarrow 0} D^{(\epsilon^{})}(t_{\epsilon^{}}+cw_{\epsilon^{}})=
\norm{\left((2\gamma)^{1-\ell}e^{-c}\hat{v}(x_0)+I^\natural_\infty\right)-I^\natural_\infty},
\end{split}
\]
and
\[
\begin{split}
\limsup\limits_{\epsilon^{}\rightarrow 0} D^{(\epsilon^{})}(t_{\epsilon^{}}+cw_{\epsilon^{}})=
\norm{\left((2\gamma)^{1-\ell}e^{-c}\tilde{v}(x_0)+I^\natural_\infty\right)-I^\natural_\infty},
\end{split}
\] 
where $\tilde{v}(x_0), \hat{v}(x_0)\in \mathrm{Bas}$, and  $\mathrm{Bas}$ denotes the basin of attraction which  is defined in \eqref{Bassin}.
On the other hand since, for  any $c\in \mathbb{R}$, 
$\lim\limits_{\epsilon\rightarrow 0}R(t_\epsilon+cw_\epsilon)=0$ (see Proposition \ref{jcppp11}),
 inequality \eqref{sand} allow us to  deduce
\[
0\leq \liminf\limits_{\epsilon\rightarrow 0} d^{(\epsilon^{})}(t_{\epsilon^{}}+cw_{\epsilon^{}})
\leq
\liminf\limits_{\epsilon\rightarrow 0}  D^{(\epsilon^{})}(t_{\epsilon^{}}+cw_{\epsilon^{}})
=
\norm{\left((2\gamma)^{1-\ell}e^{-c}\hat{v}(x_0)+I^\natural_\infty\right)-I^\natural_\infty},
\]
and
\[
\norm{\left((2\gamma)^{1-\ell}e^{-c}\tilde{v}(x_0)+I^\natural_\infty\right)-I^\natural_\infty}=\limsup\limits_{\epsilon\rightarrow 0} D^{(\epsilon^{})}(t_{\epsilon^{}}+cw_{\epsilon^{}})
\leq
\limsup\limits_{\epsilon\rightarrow 0}  d^{(\epsilon^{})}(t_{\epsilon^{}}+cw_{\epsilon^{}})
\leq 1.
\]
Using  Scheff\'e's Lemma, the following limit is obtained
\[
\lim\limits_{c\rightarrow \infty}\liminf\limits_{\epsilon\rightarrow 0} d^{(\epsilon^{})}(t_{\epsilon^{}}+cw_{\epsilon^{}})=0.
\]
Finally, recalling that $\tilde{v}(x_0)\not=0$ and using  Lemma \ref{a5}, we get
\[
\lim\limits_{c\rightarrow -\infty}\limsup\limits_{\epsilon\rightarrow 0} d^{(\epsilon^{})}(t_{\epsilon^{}}+cw_{\epsilon^{}})=1,
\]
which implies the  statement in part (ii) of Theorem \ref{teo2}. This completes the proof.
\end{proof}
Next, we prove Corollary \ref{pt}.
\begin{proof}[Proof of Corollary \ref{pt}]From the same arguments used in the proof of Theorem \ref{teo2} part (ii),  we deduce
\[
\begin{split}
\norm{\left((2\gamma)^{1-\ell}e^{-c}\hat{v}(x_0)+I^\natural_\infty\right)-I^\natural_\infty}
&\leq  \liminf\limits_{\epsilon\rightarrow 0} d^{(\epsilon^{})}(t_{\epsilon^{}}+cw_{\epsilon^{}})\\
&\hspace{-2cm}\leq
\limsup\limits_{\epsilon\rightarrow 0}  d^{(\epsilon^{})}(t_{\epsilon^{}}+cw_{\epsilon^{}})
\leq 
\norm{\left((2\gamma)^{1-\ell}e^{-c}\tilde{v}(x_0)+I^\natural_\infty\right)-I^\natural_\infty},
\end{split}
\]
where $\tilde{v}(x_0), \hat{v}(x_0)\in \mathrm{Bas}$, where $\mathrm{Bas}$ is defined in \eqref{Bassin}.  
Since $\Big|\sum\limits_{k=1}^{m}e^{i\theta_k t}v_k \Big|$ is constant,  we obtain that $|\tilde{v}(x_0)|=|\hat{v}(x_0)|$, and using
the invariance property \eqref{invarianza1} we deduce
\[
\norm{\left((2\gamma)^{1-\ell}e^{-c}\tilde{v}(x_0)+I^\natural_\infty\right)-I^\natural_\infty}=\norm{\left((2\gamma)^{1-\ell}e^{-c}\hat{v}(x_0)+I^\natural_\infty\right)-I^\natural_\infty}.
\]
Therefore,
\[
\lim\limits_{\epsilon\rightarrow 0} d^{(\epsilon^{})}(t_{\epsilon^{}}+cw_{\epsilon^{}})=\norm{\left((2\gamma)^{1-\ell}e^{-c}\tilde{v}(x_0)+I^\natural_\infty\right)-I^\natural_\infty}.
\]
\end{proof}

\subsection{Proofs for the superposition process }
\begin{proof}[Proof of Lemma \ref{opo}]
Let us first fix $\epsilon>0$.
From the L\'evy-It\^o decomposition, for each $j\geq 1$, we have 
\[
\xi^{(j)}_t=a_jt+\sqrt{\sigma_j} B^{(j)}_t+L^{(j)}_t \qquad \textrm{for}\quad t\geq 0,
\] 
where $B^{(j)}=(B^{(j)}_t: t\geq 0)$ is a standard Brownian motion and $L^{(j)}=(L^{(j)}_t: t\geq 0)$ is a pure jump L\'evy process  which is independent of $B^{(j)}$.
Therefore, for each $j\geq 1$, we deduce
\[
X^{(\epsilon,j)}_t=x_je^{-\gamma_j t}+\sqrt{\epsilon}\frac{a_j(1-e^{-\gamma_j t})}{\gamma_j}+\sqrt{\epsilon}\sqrt{\sigma_j}\int\limits_{0}^{t}{e^{-\gamma_j(t-s)}\ud B^{(j)}_s}+
\sqrt{\epsilon}\int\limits_{0}^{t}{e^{-\gamma_j(t-s)}\ud L^{(j)}_s},
\]
for any $t\geq 0$. In other words, for each $t\ge 0$, the r.v. $\chi^{(\epsilon)}_t$ is well-defined if and only if  each term
\[
\sum\limits_{j=1}^{\infty} m_jx_je^{-\gamma_j t}, \qquad   \sum\limits_{j=1}^{\infty}m_j\frac{a_j(1-e^{-\gamma_j t})}{\gamma_j}, \qquad M_t=\sum\limits_{j=1}^{\infty}m_j\sqrt{\sigma_j}\int\limits_{0}^{t}{e^{-\gamma_j(t-s)}\ud B^{(j)}_s},
\]
and 
\[
N_t=\sum\limits_{j=1}^{\infty} m_j\int\limits_{0}^{t}{e^{-\gamma_j(t-s)}\ud L^{(j)}_s},
\]
are well-defined.

The finiteness of the first two terms is clear. Indeed,  from condition \eqref{C2}, we have 
\[
\left|\sum\limits_{j=1}^{\infty}m_jx_je^{-\gamma_j t}\right|\leq
\sum\limits_{j=1}^{\infty}m_j|x_j|e^{-\gamma_j t}\leq
\sum\limits_{j=1}^{\infty}m_j|x_j|<
\infty \quad \textrm{for}\quad t\geq 0. 
\]
For the second term, we observe from the first condition of \eqref{C3} that 
\[
\left|\sum\limits_{j=1}^{\infty}m_j\frac{a_j(1-e^{-\gamma_j t})}{\gamma_j}\right|\leq
\sum\limits_{j=1}^{\infty}\frac{m_j|a_j|}{\gamma_j}{(1-e^{-\gamma_j t})}\leq
\sum\limits_{j=1}^{\infty}\frac{m_j|a_j|}{\gamma_j}<
\infty  \quad \textrm{for}\quad t\geq 0.
\]
For the continuous local martingale term $M:=(M_t,t\geq 0)$, we use its quadratic variation to deduce that $M_t$ is well-defined
 if and only if
\[
\sum\limits_{j=1}^{\infty}
\frac{m^2_j \sigma_j}{2\gamma_j}(1-e^{-2\gamma_j t})<\infty.
\]
The latter is finite if the second condition in \eqref{C3} holds, implying the $M_t$ is well-defined for any $t\ge 0$.

Finally, we analyse the pure jump term.  In order to deduce that the r.v. $N_t$, which is infinitely divisible,  is well-defined for any $t\geq 0$, we need to verify that its characteristic function is also well-defined. In other words,  we need to verify that 
\[
\sum\limits_{j=1}^{\infty}\int\limits_{0}^{t}\psi_j(e^{-\gamma_js}zm_j)\ud s \quad \textrm{ exists for }\quad   z\in \mathbb{R},
\]
where  $\psi_j$ denotes the characteristic exponent of $L^{(j)}$, for $j\ge 1$. In order to do so, we first observe that each 
\[
\int\limits_{0}^{t}{e^{-\gamma_j (t-s)}\ud L^{(j)}_s},
\]
is infinitely divisible with characteristics  $(\lambda^{(j)}_t,0,\nu^{(j)}_t)$, where
\[
\lambda^{(j)}_t=\int\limits_{\mathbb{R}}\pi_j(\ud x)\int\limits_{0}^{t}e^{-\gamma_j s}x\left(\mathbf{1}_{\{|e^{-\gamma_j s }x|\le 1\}}-1_{\{|x|\le 1\}}\right)\ud s \quad\textrm{and}\quad\nu^{(j)}_t(B)=\int\limits_0^t \pi_j(e^{\gamma_j s}B)\ud s,
\]
for any $B\in \mathcal{B}(\mathbb{R})$.
Therefore,
\[
\int\limits_{0}^{t}\psi_{j}(e^{-\gamma_js}zm_j)\ud s= izm_j\lambda^{(j)}_t+\int\limits_{\mathbb{R}}\left(e^{im_jzx}-1-im_jzx\mathbf{1}_{\{|x|\le 1\}}\right)\nu^{(j)}_t(\ud x).
\]
Since
\begin{equation*}\label{taylor}
\left|e^{izx}-1-izx\mathbf{1}_{\{|x|\le 1\}}\right|\leq \frac{1}{2}z^2x^2 \mathbf{1}_{\{|x|\le 1\}}+2 \mathbf{1}_{\{|x|> 1\}}\qquad  \text{ for } x,z\in \mathbb{R}.
\end{equation*}
Similar  computations as those used in  the proof of Theorem 17.5 in \cite{Sa} allow us to  deduce that for any $z\in \mathbb{R}$
\[
\begin{split}
\left|\sum\limits_{j=1}^{\infty} \int\limits_{\mathbb{R}}\left(e^{im_jzx}-1-im_jzx\mathbf{1}_{\{|x|\le 1\}}\right)\nu^{(j)}_t(\ud x)\right | &\leq
\frac{z^2}{4}\sum\limits_{j=1}^{\infty} \frac{m^2_j}{\gamma_j}\int\limits_{\mathbb{R}}(1\wedge x^2)\pi_j(\ud x)\\
&\hspace{2cm}+
2\sum\limits_{j=1}^{\infty} \frac{1}{\gamma_j}\int\limits_{\{|x|>1\}}\ln(|x|)\pi_j(\ud x),
\end{split}
\]
and
\[
\left|\sum\limits_{j=1}^{\infty} zm_j\int\limits_{\mathbb{R}}\pi_j(\ud x)\int\limits_{0}^{t}e^{-\gamma_j s}x\left(\mathbf{1}_{\{|e^{-\gamma_j s }x|\le 1\}}-1_{\{|x|\le 1\}}\right)\ud s\right| \leq 
|z| \sum\limits_{j=1}^{\infty} \frac{m_j}{\gamma_j}\pi_j([-1,1]^c),
\]
where the left-hand sides of both inequalities are finite by assumptions \eqref{nearzero} and  \eqref{nearzero1}. It is important to note that all our bounds do not depend on $t$, implying that  
$\chi^{(\epsilon)}_t$ converges in distribution as $t$ goes to infinity to $\chi^{(\epsilon)}_{\infty}$, where 
\[
\chi^{(\epsilon)}_{\infty}=
\sqrt{\epsilon}\sum\limits_{j=1}^{\infty}\frac{m_ja_j}{\gamma_j}+
\sqrt{\epsilon}\sum\limits_{j=1}^{\infty}m_j\sqrt{\sigma_j}\int\limits_{0}^{\infty}{e^{-\gamma_j s }\ud B^{(j)}_s}+
\sqrt{\epsilon}\sum\limits_{j=1}^{\infty}m_j\int\limits_{0}^{\infty}{e^{-\gamma_j s}\ud L^{(j)}_s},
\]
and  does not depend on the initial configuration $x_0$.
Moreover from its structure, it is not difficult to see that   $\chi^{(\epsilon)}_{\infty}$ is self-decomposable.
\end{proof}
\begin{proof}[Proof of Theorem \ref{supou}] Our arguments follows from similar reasonings as in the proof of the main result. For simplicity on exposition, we denote by $\mu^{\natural, m}_t$  for  the distribution of $I^{\natural, m}_t$, for $t\ge 0$, and $\mu^{\natural,m}_\infty$ for the distribution of $I^{\natural,m}_\infty$. Similarly, for $j\ge 1$, and $t\ge 0$,  we write $\mu^{(\natural, )j}_t$ for the distribution of  $I^{(\natural,  j)}_t$.

We now assume that there exist $j \ge 1$ such that for any $t\ge 0$, $\mu^{(\natural, j)}_t$ satisfies condition {\bf (H)} and prove that this implies that for any $t>0$,
$\widehat{\mu^{\natural,m}}_t(\cdot)$  is integrable and 
\[
\lim_{R \to \infty} \sup_{s> t_0(R)}\int_{[-R, R]^c}|\widehat{\mu^{\natural,m}}_s(\lambda)| \ud \lambda =0, 
 \]
 where $\widehat{\mu^{\natural,m}_t}$ denotes the characteristic function of $\mu^{\natural,m}_t$ and $t_0(R)$ is  positive and goes to $\infty$ as $R$ increases.  We recall  the latter integrability condition on $\widehat{\mu^{\natural,m}_t}$ implies that $I^{\natural, m}_t$,  $I^{\natural, m}_\infty$,  $\chi^{(1)}_t$ and its invariant distribution $\mu^{(1, m)}$ posses continuous densities. 
 
Without loss of generality, we assume that for any $t\ge 0$, $\mu^{(\natural, 1)}_t$ satisfies condition {\bf (H)}. Since the L\'evy processes $(\xi^{(j)}, j\ge 1)$ are independent, we deduce 
\begin{equation}\label{cochinita}
\left|\widehat{\mu^{\natural,m}_t} (\lambda)\right|=\prod_{j=1}^{\infty}\left|\widehat{\mu^{(\natural,j)}_t}(m_j\lambda)\right |\le \left|\widehat{\mu^{(\natural, 1)}_t}(m_1\lambda)\right|, 
\end{equation}
 implying that
  \[
  \int_{\mathbb{R}}\left|\widehat{\mu^{\natural,m}_t} (\lambda)\right|\ud \lambda\le \frac{1}{m_1} \int_{\mathbb{R}}\left|\widehat{\mu^{(\natural,1)}_t} (u)\right|\ud u<\infty.
  \]
 Next, we use again inequality \eqref{cochinita} and deduce 
 \[
  \int_{[-R, R]^c}\left|\widehat{\mu^{\natural,m}_t} (\lambda)\right|\ud \lambda\le \frac{1}{m_1} \int_{[-Rm_1, R m_1]^c}\left|\widehat{\mu^{(\natural,1)}_t} (u)\right|\ud u\le \frac{1}{m_1} \int_{[-R, R ]^c}\left|\widehat{\mu^{(\natural,1)}_t} (u)\right|\ud u,
 \]
 where the last inequality follows from the fact that $m_1\le 1$. In other words, 
 \[
\lim_{R \to \infty} \sup_{s> t_0(R)}\int_{[-R, R]^c}|\widehat{\mu^{\natural,m}}_s(\lambda)| \ud \lambda \le \frac{1}{m_1} \lim_{R \to \infty} \sup_{s> t_0(R)}\int_{[-R, R ]^c}\left|\widehat{\mu^{(\natural,1)}_t} (u)\right|\ud u=0,
 \]
  which implies that $\mu^{\natural,m}_t$ satisfies condition {\bf (H)}.
  
For the sequel, we follow the same notation as in the proof of Lemma \ref{opo} and observe 
$I^{\natural, m}_t=M_t+N_t,$ for  $t\ge 0$. 
From our assumptions, its limiting distribution $I^{\natural, m}_\infty$ is well-defined and moreover the r.v.
\[
\sqrt{\epsilon}\sum\limits_{j=1}^{\infty}\frac{m_ja_j}{\gamma_j}+
\sqrt{\epsilon}
I^{\natural, m}_\infty,
\]
has the same distribution as $\mu^{(\epsilon, m)}$.  For each $t> 0$,  we introduce
\[D^{(\epsilon, m)}(t):=\norm{\left(\frac{1}{\sqrt{\epsilon}}\sum\limits_{j=1}^{\infty}m_j x_je^{-\gamma_j t}
+ 
I^{\natural, m}_\infty\right)
-I^{\natural, m}_\infty
},
\]
and
\[R^{(m)}(t):=\norm{\left(-\sum\limits_{j=1}^{\infty}\frac{m_ja_j}{\gamma_j}e^{-\gamma_j t}+I^{\natural, m}_t\right)
-I^{\natural, m}_\infty
}.\]
The same argument  as those used in the proof of Lemma \ref{lemmadif} allows us to deduce
\[
\left|d^{(\epsilon, m)}(t)-D^{(\epsilon. m)}(t)\right|\leq R^{(m)}(t) \qquad \text{ for  }\quad t>0.
\]
Recall that $J=\{j\geq 1: \gamma_j=\inf\limits_{k\in \mathbb{N}}{\gamma_k}\}\not= \emptyset$. For simplicity, we define  
\[
v^{(\epsilon)}(t):=\frac{1}{\sqrt{\epsilon}}\sum\limits_{j=1}^{\infty}m_jx_{j}e^{-\gamma_j t}=\frac{1}{\sqrt{\epsilon}}\sum\limits_{j\in J}m_jx_{j}e^{-\gamma_j t}+\frac{1}{\sqrt{\epsilon}}\sum\limits_{j\in \mathbb{N}\setminus J}m_jx_{j}e^{-\gamma_j t},
\]
and observe that for any $b\in \mathbb{R}$, we have
\[
\lim\limits_{\epsilon\rightarrow 0}v^{(\epsilon)}(t_{\epsilon}+bw_\epsilon):=e^{-b}\sum\limits_{j\in J}m_jx_j,
\]
where $t_\epsilon$ and $w_\epsilon$ are chosen as in the statement. Therefore,
\[
\lim\limits_{\epsilon\rightarrow 0}D^{(\epsilon, m)}(t_{\epsilon}+bw_\epsilon):=
\norm{\left(e^{-b}\sum\limits_{j\in J}m_jx_j
+ I^{\natural, m}_\infty \right)-I^{\natural, m}_\infty
}.\]
It remains to prove that $R^{(m)}(t)$ goes to $0$ as $t$ goes to infinity.
Using the triangle inequality and Lemma \ref{sca} part i), we deduce
\begin{align*}
R^{(m)}(t) &\le \norm{I^{\natural, m}_t
-I^{\natural, m}_\infty
}+\norm{\left(-\sum\limits_{j=1}^{\infty}\frac{m_ja_j}{\gamma_j}e^{-\gamma_j t}+I^{\natural, m}_\infty\right)
-I^{\natural, m}_\infty
}.
\end{align*}
Since $\mu^{\natural,m}_t$ satisfies condition {\bf (H)} for any $t> 0$, we have that  $I^{\natural, m}_\infty$ has  a continuous  density  and since
\[
\lim\limits_{t\rightarrow \infty}\sum\limits_{j=1}^{\infty}\frac{m_ja_j}{\gamma_j}e^{-\gamma_j t}=0,
\] 
an application of   Scheff\'e's Lemma  allows us to deduce
\[
\lim_{t\to 0}\norm{\left(-\sum\limits_{j=1}^{\infty}\frac{m_ja_j}{\gamma_j}e^{-\gamma_j t}+I^{\natural, m}_\infty\right)
-I^{\natural, m}_\infty}=0.
\]
The proof that $\norm{I^{\natural, m}_t
-I^{\natural, m}_\infty
}$ goes to $0$, as $t$ increases,  follows from  the same arguments as those used in the proof of Proposition \ref{jcppp11} given that $\mu^{\natural,m}_t$ satisfies condition {\bf (H)}.
\end{proof}

\subsection{Proof of Theorem \ref{ave}}
\begin{proof}[Proof of the Theorem \ref{ave}] Let us consider the sequence $(\xi^{(j)}, j\ge 1)$ of independent copies of the stable process $\xi$ with drift $a$. Therefore, for $n\ge 1$ and $j\in\{1,\ldots, n\}$,  we can write
\[
X^{(\epsilon_n),j}_t=x_0e^{-\gamma t}+\sqrt{\epsilon_n}\frac{a(1-e^{-\gamma t})}{\gamma}+\sqrt{\epsilon_n}\int_{0}^{t}e^{-\gamma(t-s)}\ud  \tilde{\xi}^{(j)}_s,
\]
where $\tilde\xi^{(j)}_t=\xi^{(j)}_t-a t$, for $t\ge 0$.

For simplicity on exposition, for each $j\ge 1$, we denote
\[
Y^{(j)}_t:=\int_{0}^{t}e^{-\gamma(t-s)}\ud  \tilde\xi^{(j)}_s \qquad\textrm{for} \quad t\ge 0.
\]
In other words, the average process $(A^{(n)}_t,t\geq 0)$ satisfies
\[
A^{(n)}_t=x_0e^{-\gamma t}+\sqrt{\epsilon_n}\frac{a(1-e^{-\gamma t})}{\gamma}+
\frac{\sqrt{\epsilon_n}}{n}\sum\limits_{j=1}^{n}Y^{(j)}_t.
\]
Observe that the sequence of processes $(Y^{(j)}, j\ge 1)$ defined above is clearly independent and identically distributed. Moreover, it is not difficult to deduce that for each $t>0$, the distribution of $Y^{(j)}_t$ is strictly stable with characteristic exponent
\[
 \psi_{t,\alpha}(z)= -\frac{c(1-e^{-\alpha\gamma t})}{\alpha\gamma}|z|^{\alpha}\left(1-i\beta\tan(\pi\alpha/2)\mathrm{sgn}(z)\right) \qquad \textrm{for}\quad z\in \mathbb{R}.
\]
Moreover,  for each $j\ge 1$,  the following limit
\[
Y^{(j)}_\infty:=\lim_{t\rightarrow \infty}{Y^{(j)}_t}\qquad \textrm{exists,} 
\]
and is also a stable distribution with characteristic exponent
\[
 \psi_{\infty,\alpha}(z)= -\frac{c}{\alpha\gamma}|z|^{\alpha}\left(1-i\beta\tan(\pi\alpha/2)\mathrm{sgn}(z)\right) \qquad \textrm{for}\quad z\in \mathbb{R}.
\]

Next, for each $t> 0$, we define the auxiliary metric
\[D^{(n)}(t):=\norm{\left(\frac{n}{\sqrt{\epsilon_n}}x_0e^{-\gamma t}-\frac{n}{\gamma}ae^{-\gamma t}+\sum\limits_{j=1}^{n} Y^{(j)}_\infty\right)
-\left(\sum\limits_{j=1}^{n} Y^{(j)}_\infty\Big)\right)
},\]
and the error term
\[R^{(n)}(t):=\norm{\left(\sum\limits_{j=1}^{n}Y^{(j)}_t\right)-\left(\sum\limits_{j=1}^{n}Y^{(j)}_\infty\right)}.
\]
Similar reasonings as those used in Lemma \ref{errores} allow us to deduce
\[
\left|d^{(n)}(t)-D^{(n)}(t)\right|\leq R^{(n)}(t)\qquad  \text{ for  } t>0.
\]
On the other hand by the scaling   property  of the total variation distance (see Lemma \ref{sca} part  (ii)), we deduce
\[D^{(n)}(t)=\norm{\left(\frac{n^{1-1/\alpha}}{\sqrt{\epsilon_n}}x_0e^{-\gamma t}-\frac{n^{1-1/\alpha}}{\gamma}ae^{-\gamma t}+n^{-1/\alpha} \sum\limits_{j=1}^{n} Y^{(j)}_\infty\right)
-\left(n^{1-1/\alpha} \sum\limits_{j=1}^{n} Y^{(j)}_\infty\right)},
\]
and
\[R^{(n)}(t):=\norm{n^{-1/\alpha} \left(\sum\limits_{j=1}^{n}Y^{(j)}_t\right)-n^{-1/\alpha} \left(\sum\limits_{j=1}^{n}Y^{(j)}_\infty\right)}.
\]
Moreover, since the sequence $(Y^{(j)}, j\ge 1)$ is independent and with the same distribution, we have for each $t>0$
\[
n^{-1/\alpha} \sum\limits_{j=1}^{n} Y^{(j)}_t\overset{(d)}=\left(\frac{1-e^{-\gamma \alpha t}}{\alpha\gamma}\right)^{1/\alpha} \mathcal{S}_\alpha,
\]
where $\overset{(d)}=$ means identity in law or distribution. We observe that the latter identity in law also holds for $t=\infty$. 

In other words, we can rewrite the distance $D^{(n)}$ (after using the scaling   property  of the total variation distance) and error term $R^{(n)}$ as follows,
\[
D^{(n)}(t)=\norm{\left(\frac{n^{1-1/\alpha}}{\sqrt{\epsilon_n}(\alpha\gamma)^{1/\alpha} }x_0e^{-\gamma t}-\frac{n^{1-1/\alpha}}{\gamma (\alpha\gamma)^{1/\alpha} }ae^{-\gamma t}+ \mathcal{S}_\alpha\right)
- \mathcal{S}_\alpha},
\]
and
\[R^{(n)}(t)=\norm{\left(\frac{1-e^{-\gamma \alpha t}}{\alpha\gamma}\right)^{1/\alpha} \mathcal{S}_\alpha-\left(\frac{1}{\alpha\gamma}\right)^{1/\alpha} \mathcal{S}_\alpha}.
\]
Finally, we take the sequences $t_n$ and $w_n$ as in the statement and recall that  $\mathcal{S}_\alpha$ has a continuous unimodal  density. Therefore an application of  Scheff\'e's Lemma  allow us to deduce 
\[
\lim_{n\to\infty}D^{(n)}(t_n+cw_n)=\norm{\left( e^{-c}x_0+ \mathcal{S}_\alpha\right)
- \mathcal{S}_\alpha}, 
\]
and
\[
\lim_{n\to\infty}R^{(n)}(t_n+cw_n)=0.
\]
This completes the proof of our result.
\end{proof}

\appendix
\section{Tools}\label{tools}
The following section contains useful properties that help us to make this article more fluid.
Since  almost all  proofs are straightforward, we left most of the details to the interested  reader except for those that seem to be not so direct. 
\begin{lemma}\label{sca}
Let $(\Omega,\mathcal{F},\mathbb{P})$ be a probability space.
Let $X,Y:\Omega\rightarrow \mathbb{R}^d$ be random variables such that their laws are 
 are absolutely continuous with respect to the Lebesgue measure on $(\mathbb{R}^d,\mathcal{B}(\mathbb{R}^d))$. For any $a,b\in \mathbb{R}^d$ and $c\in \mathbb{R}\setminus\{0\}$,
the following holds true:
\begin{itemize}
\item[i)] 
$
\norm{(X+a)-(Y+b)}=\norm{(X+a-b)-Y}.
$
\item[ii)] 
$
\norm{(cX)-(cY)}=\norm{X-Y}.
$
\item[iii)] $
\norm{(cX+a)-(cY)}=\norm{(X+\nicefrac{a}{c})-Y}.
$
\end{itemize}
\end{lemma}
\begin{proof}
The idea of the proofs of (i)-(iii) follow by the Change of Variable Theorem and using the characterisation of the total variation distance between two probabilities with densities
\[
\norm{X-Y}=\frac{1}{2}\int\limits_{\mathbb{R}^d}{\left|f_X(z)-f_Y(z)\right|\ud z},
\] 
where $f_X$ and $f_Y$ are the densities of $X$ and $Y$, respectively.
\end{proof}

\begin{lemma}[Convolution]\label{ayu5}
Let $(\Omega,\mathcal{F},\mathbb{P})$ be a probability space and
 $X_1,X_2,Y_1,Y_2,Z$ be r.v.'s defined on $\Omega$ and taking values in $ \mathbb{R}^d.$   
\begin{itemize}
\item[i)] Assume that $X_1$ and $X_2$ are independent, and that 
 $Y_1$ and $Y_2$ are independent. Then
\[
\norm{\left(X_1+X_2\right)-\left(Y_1+Y_2\right) }\leq   \norm{X_1-Y_1 }+
\norm{X_2-Y_2 }.
\]
\item[ii)] Assume that $(X_1, Y_1)$ is independent of  $Z$. Then
\[
\norm{\left(X_1+Z\right)-\left(Y_1+Z\right)}\leq  
\norm{X_1-X_2}.
\]
\end{itemize}
\end{lemma}
\begin{proof}
The idea of the proof follows from  the fact that the distribution of the sum of two independent random variables corresponds to their convolution. 
\end{proof}

The following Lemma is stated  in \cite{BJ1} and its proof is straightforward, we leave the details to the interested reader.
\begin{lemma}\label{a5}
Let $\left(\Omega,\mathcal{F},\mathbb{P}\right)$ be a probability space and let $X:\Omega \to \mathbb{R}^d$ be a random variable.
Assume that $\mathcal{L}(X)$ is absolutely continuous with respect to the Lebesgue measure on $(\mathbb{R}^d,\mathcal{B}(\mathbb{R}^d))$. Let $(\alpha_{\epsilon},\epsilon>0)$ be a function such that $\lim\limits_{\epsilon\rightarrow 0}{\|\alpha_{\epsilon}\|}=\infty$. Then
$
\lim\limits_{\epsilon\rightarrow 0}\norm{\left(\alpha_\epsilon+X\right)-X}=1.
$
\end{lemma}

\begin{proposition}[Real Spectrum]\label{lyapunov}
Let $Q\in \mathcal{M}^+(d)$ and $x_0\in \mathbb{R}^d\setminus\{0\}$. 
Denote by $\gamma_1,\gamma_2,\ldots,\gamma_d$ the eigenvalues of $Q$ which are assumed to be real.
\begin{itemize}
\item[i)] If  $Q$ is symmetric, then there exist $\gamma:=\gamma(x_0)>0$ and  $v(x_0)\in \mathbb{R}^d\setminus\{0\}$  such that
\begin{equation*}
\lim\limits_{t\rightarrow \infty}e^{\gamma t}e^{-  tQ }x_0=v(x_0)\in \mathbb{R}^{d}\setminus \{ 0\}.
\end{equation*}
\item[ii)] If $Q$ is not symmetric, then  there exist $\gamma:=\gamma(x_0)>0$, $l:=l(x_0)\in \{1,2,\ldots,d\}$ and $v(x_0)\in \mathbb{R}^d\setminus\{0\}$  such that
\begin{equation*}
\lim\limits_{t\rightarrow \infty}\frac{e^{\gamma t}e^{-  tQ }x_0}{t^{l-1}}=v(x_0)\in \mathbb{R}^{d}\setminus \{ 0\}.
\end{equation*}
\end{itemize}
\end{proposition}
\begin{proof}
We first prove part (i). Without loss of generality, 
we assume that $\gamma_1\leq \cdots \leq \gamma_d$.
Since $Q$ is a symmetric matrix, then it is diagonalisable. In other words, there exist orthogonal matrix $U$ such that
$Q=U\diag(\gamma_1,\ldots,\gamma_d)U^{T}$.  Therefore,  if we take  $x_0\not=0$ and let $y=U^{T}x_0=(y_1,\ldots,y_d)^{T}\not=0$, we have
\[e^{-tQ}x_0=U\diag(e^{-\gamma_1 t}y_1 ,\ldots,e^{-\gamma_d t}y_d) \qquad \textrm{for}\quad t\ge 0.\] 
We define  $\tau(x_0):=\min\{j\in\{1,\ldots,d\}:y_j \not=0\}$, and take limit as $t$ increases in the previous identity to deduce 
\[
\lim\limits_{t\rightarrow \infty}e^{\gamma_{\tau(x_0)}t}e^{-tQ}x_0=U\diag(0_{\tau(x_0)-1} ,y_{\tau(x_0)},\ldots,y_{k(x_0)} ,0_{d-k(x_0)})\not=0,
\]
where $k(x_0):=\max\{j\geq \tau(x_0):\gamma_j=\gamma_{\tau(x_0)}\}$,  and $0_{\tau(x_0)-1}$, $0_{d-k(x_0)}$  are the zeros of $\mathbb{R}^{\tau(x_0)-1}$ and $\mathbb{R}^{d-k(x_0)}$, respectively.  The latter  implies the statement of part (i).

For the proof of part (ii), we observe that since $Q$ is not symmetric,  it is not always diagonalisable.
Nevertheless, $Q$ has a Jordan decomposition.  In other words, there exist an invertible $d$-square matrix $U$ and a $d$-square  matrix $J$ such that
$Q=UJU^{-1}$, where 
\begin{equation}\label{jord}
J=
\left[\begin{array}{ccccc}
 J_{k_1}(\gamma_1)&0&\cdots &0 & 0\\
0& J_{k_2}(\gamma_2)&\cdots &0& 0\\
\vdots &\vdots &\ddots &\vdots & \vdots \\
0&0&\cdots & J_{k_{m-1}}(\gamma_{m-1})& 0\\
0&0&\cdots &0& J_{k_{m}}(\gamma_{m})\\
\end{array}\right],
\end{equation}
 $m\leq d$ and $\gamma_1,\ldots,\gamma_m\in (0,\infty)$ are the eigenvalues of $Q$.
In order to make the proof more readable, we first explain our arguments  in a Jordan block.  Recall that a Jordan block $J_k(\gamma)$ is a $k$-square upper triangular matrix of dimension $k$ defined as follows
\begin{displaymath}
J_k(\gamma)=\left[\begin{array}{ccccc}
\gamma&1&\cdots &0 & 0\\
0&\gamma&\cdots &0& 0\\
\vdots &\vdots &\ddots &\vdots & \vdots \\
0&0&\cdots &\gamma& 1\\
0&0&\cdots &0& \gamma\\
\end{array}\right].
\end{displaymath}
Since any Jordan block can be written as a sum of diagonal matrix and a nilpotent matrix
then the exponential matrix of a Jordan block can be computed explicitly
\begin{displaymath}
e^{-tJ_k(\gamma)}=\left[\begin{array}{ccccc}
e^{-\gamma t}&te^{-\gamma t}&\cdots & \frac{t^{k-2}}{(k-2)!}e^{-\gamma t} & \frac{t^{k-1}}{(k-1)!}e^{-\gamma t}\\
0&e^{-\gamma t}&\cdots & \frac{t^{k-3}}{(k-3)!}e^{-\gamma t} & \frac{t^{k-2}}{(k-2)!}e^{-\gamma t} \\
\vdots &\vdots &\ddots &\vdots & \vdots \\
0&0&\cdots &e^{-\gamma t}& te^{-\gamma t}\\
0&0&\cdots &0& e^{-\gamma t}\\
\end{array}\right],
\end{displaymath}
for any $t\ge 0$.

Next, we consider the simplest case which is  $Q=UJ_d(\gamma)U^{-1}$, where $\gamma>0$. Taking into account the previous identity,  we have

\begin{displaymath}
e^{-tQ}=U\left[\begin{array}{ccccc}
e^{-\gamma t}&te^{-\gamma t}&\cdots & \frac{t^{d-2}}{(d-2)!}e^{-\gamma t} & \frac{t^{d-1}}{(d-1)!}e^{-\gamma t}\\
0&e^{-\gamma t}&\cdots & \frac{t^{d-3}}{(d-3)!}e^{-\gamma t} & \frac{t^{d-2}}{(d-2)!}e^{-\gamma t} \\
\vdots &\vdots &\ddots &\vdots & \vdots \\
0&0&\cdots &e^{-\gamma t}& te^{-\gamma t}\\
0&0&\cdots &0& e^{-\gamma t}\\
\end{array}\right]U^{-1}, 
\end{displaymath}
for any $t\ge 0.$ We let $x_0\in \mathbb{R}^d\setminus\{0\}$ and write  $y=U^{-1}x_0=(y_1,\ldots,y_d)^T\not=0$. 
We define  $\tau(\tilde{y})=\max\{j\in \{1,\ldots,d\}:y_j\not=0\}$ and observe
\begin{displaymath}
\lim\limits_{t\rightarrow \infty}\frac{e^{\gamma_{\tau(y)} t}}{t^{{\tau(y)}-1}}e^{-tQ}x_0=
U\left[\begin{array}{c}
\frac{1}{(\tau(y)-1)!}y_{{\tau(y)}}\\
0\\
\vdots \\
0
\end{array}\right]\not=0.
\end{displaymath}
which implies our result.

For the general case, we only provide the main ideas and leave the details to the interested reader.  Notice that 
\[e^{-tQ}x_0=Ue^{-tJ}y \;\;\quad \textrm{for }\quad t\geq 0,
\]
where $y=U^{-1}x_0\not=0$ and $J$ is given  in \eqref{jord} and  where $J_{k_1},\ldots,J_{k_m}$ denote the Jordan blocks. For simplicity of exposition, we  denote by $y=(y_1,\ldots,y_d)^{\mathrm{T}}$ the coordinates of the vector $y$.

Next, we let $r_0=0$ and consider the partial sums
$r_j=\sum\limits_{i=1}^{j}{k_j}$ for each $j\in \{1,\ldots,m\}$.
For each $j\in \{1,\ldots,m\}$, we define the $k_j$-dimensional vector
$[y]_j:=(y_{r_{j-1}+1},\ldots,y_{r_{j}})^{\mathrm{T}}$ and let $I:=\{j\in\{1,\ldots,m\}: e^{-tJ_{k_j}}[y]_j\not=0\}$ which is not empty. Proceeding similarly as above  in each Jordan block, we observe that 
for any $j\in I$ there exist $\gamma_j>0$ and $l_j\in \{1,\ldots,d\}$ such that 
\[
\lim\limits_{t\rightarrow \infty}\frac{e^{\gamma_j t}}{t^{l_j-1}}e^{-tJ_{k_j}}[y]_j\not=0.
\]
Then, we take $\gamma:=\min\limits_{j\in I}\gamma_j$
and define $l:=\max\limits_{j\in \tilde{I}}l_j$ where $\tilde{I}:=\{j\in I: \gamma_j=\gamma\}$.
The previous identity may lead to  \[
\lim\limits_{t\rightarrow \infty}\frac{e^{\gamma{} t}}{t^{l-1}}e^{-tJ}y\not=0,
\] and consequently
$\lim\limits_{t\rightarrow \infty}\frac{e^{\gamma{} t}}{t^{l-1}}e^{-tQ}x_0\not=0$.
\end{proof}
\noindent \textbf{Acknowledgements.} Both authors  acknowledge support from  the Royal Society and CONACyT-MEXICO (CB Research grant). This work was concluded whilst JCP was on sabbatical leave holding a David Parkin Visiting Professorship  at the University of Bath, he gratefully acknowledges the kind hospitality of the Department and University.
GB gratefully acknowledges support from a post-doctorate CONACyT-MEXICO grant held at the Department of Probability and Statistics, CIMAT.
He also gratefully acknowledges support from a post-doctorate Pacific Institute for the Mathematical Sciences (PIMS, $2017$-$2019$) grant held at the Department of Mathematical and Statistical Sciences at University of Alberta and he would like to express his gratitude to University of Alberta for all
the facilities used along the realization of this work.

\bibliographystyle{amsplain}

\end{document}